\documentclass[amssymb,amsfonts,refcheck,12pt,verbatim,righttag]{amsart}

\usepackage{amsfonts}
\usepackage{graphicx}
\usepackage{color}
\usepackage{amssymb}
\usepackage{cancel}

\usepackage{soul}

 \numberwithin{equation}{section}
\theoremstyle{plain}

\newtheorem{thm}{Theorem}[section]
\newtheorem{lem}[thm]{Lemma}
\newtheorem{pro}[thm]{Proposition}
\newtheorem{cor}[thm]{Corollary}
\newtheorem{ex}[thm]{Example}
\newtheorem{de}[thm]{Definition}
\newtheorem{rem}[thm]{Remark}

\def\Proof{\noindent {\bf Proof.\ }}

\newtheorem{ques}[thm]{Question}

\def\R {{\Bbb R}}
\def\N {{\Bbb N}}
\def\Z {{\Bbb Z}}
\def\cal{\mathcal}
\def\M {{\mathcal M}}

\def\Q {{\Bbb Q}}
\def\J {{\mathcal J}}
\def\F {{\mathcal F}}
\def\I {{\mathcal I}}
\def\A {{\mathcal A}}

\def\L {\mathcal{L}}
\def\D{\mathcal{D}}
\def\U{{\mathcal U}}

\def\x{{\bf x}}

\def\bx{{\bf x}}

\def\T{\top}

\def \S{\mathcal{S}}
\def \b0{\mathbf{0}}
\DeclareMathOperator{\topo}{\mathrm{top}}

\newcommand{\gtsim}{\succcurlyeq}
\newcommand{\ltsim}{\preccurlyeq}

\def\bM{{\bf M}}
\newcommand{\interior}[1]{%
  {\kern0pt#1}^{\mathrm{o}}%
}

\title{Lyapunov Exponents for products of  matrices}
\author{De-Jun FENG}
\address{
Department of Mathematics\\
The Chinese University of Hong Kong\\
Shatin,  Hong Kong\\
}
\email{djfeng@math.cuhk.edu.hk}

\author{Chiu-Hong Lo}
\address{
Department of Mathematics\\
The Chinese University of Hong Kong\\
Shatin,  Hong Kong\\
}
\email{chlo@math.cuhk.edu.hk}

\author{Shuang Shen}
\address{
Department of Mathematics\\
The Chinese University of Hong Kong\\
Shatin,  Hong Kong\\
}
\email{gjdyyss@163.com}


\date{}

\begin{document}

\thanks{}

\begin{abstract}
Let ${\bf M}=(M_1,\ldots, M_k)$ be a tuple of  real $d\times d$  matrices. Under certain irreducibility assumptions, we give checkable criteria for deciding  whether  ${\bf M}$ possesses the following property: there exist two constants $\lambda\in \R$ and $C>0$ such that for any $n\in \mathbb{N}$ and any $i_1, \ldots, i_n \in \{1,\ldots, k\}$, either $M_{i_1} \cdots M_{i_n}={\bf 0}$ or $C^{-1} e^{\lambda n}  \leq \| M_{i_1} \cdots M_{i_n} \| \leq C e^{\lambda n}$, where $\|\cdot\|$ is a matrix norm. The proof is based on  symbolic dynamics and  the thermodynamic formalism for matrix products. As applications, we are able to check  the absolute continuity of a class of overlapping self-similar measures on $\R$, the absolute continuity of certain self-affine measures in $\R^d$ and  the dimensional regularity of a class of sofic affine-invariant sets in the plane.

\end{abstract}

\maketitle

\section{Introduction}

In this paper, we consider  Lyapunov exponents of  matrix products. Let ${\bM}=(M_1,\ldots, M_k)$ be a given tuple of real $d\times d$ matrices. 

\begin{de}
We say that $\bM$ has a {\it uniform Lyapunov exponent modulo $0$} if there exist $C>0$ and $\lambda\in \R$ such that for any $n\in \N$ and any $i_1,\ldots, i_n\in \{1,\ldots, k\}$,
\begin{equation}
\label{e-ULE}
\mbox{ either }\quad  M_{i_1}\cdots M_{i_n}={\bf 0} \quad \mbox{ or }\quad  C^{-1}e^{\lambda n}\leq  \|M_{i_1}\cdots M_{i_n}\|\leq Ce^{\lambda n},
\end{equation}
where $\|\cdot\|$ is a given matrix norm. Clearly the above property is independent of the choice of matrix norm. 
\end{de}

\begin{de}
\begin{itemize} \item[(i)]
${\bM}$ is said to be irreducible if there is no non-zero proper linear subspace $V$ of $\R^d$  such that $M_iV\subset V$ for all $1\leq i\leq k$.
\item[(ii)]  ${\bM}$ is said to be positively irreducible if $M_i$ are all non-negative matrices and there exists  $\ell\in \N$ so that $\sum_{j=1}^\ell (\sum_{i=1}^kM_i)^j$ is a strictly positive matrix.
    \end{itemize}
\end{de}
 We remark that the positive irreducibility does not imply the irreducibility.  The main problem  we address in this paper  is  the following.

\begin{ques}
\label{ques-1}
 Suppose that $\bM$ is irreducible or positively irreducible.  Can we determine whether $\bM$ has a uniform Lyapunov exponent modulo $0$?
\end{ques}

We remark that without any irreducibility assumption, there is no general algorithm to check whether $\bM$ has a uniform Lyapunov exponent modulo $0$. This follows from the result of Blondel and Tsitsiklis  \cite{BlondelTsitsiklis2000} that the boundedness  of a matrix semigroup is generally undecidable. For details, see Section~\ref{S-8}.

Whilst Question \ref{ques-1} is of independent interest, our study is directly motivated by several questions arising in fractal geometry and dynamical systems, although their answers have been known or partially known. One is on the absolute continuity of a class of overlapping self-similar measures on $\R$,  another one
is on the
absolute continuity of certain  self-affine measures on $\R^d$,  and the last one is on
the dimensional regularity of certain sofic affine-invariant sets on the $2$-torus ${\Bbb T}^2$.
Below we describe them in more details.

First we state the question on self-similar measures. Let  $\{S_j\}_{j=1}^m$ be a family of contractive similitudes on $\R$ given by
\begin{equation}
\label{e-IFS}
S_j(x) = \rho x + b_j, \quad  j = 1, \ldots, m,
\end{equation}
where $m\geq 2$, $0<\rho<1$ and $b_1<\cdots<b_m$.
Given a probability vector $(p_1,\ldots, p_m)$, let $\mu$ be the self-similar measure generated by $\{S_j\}_{j=1}^m$ and $(p_1,\ldots, p_m)$. That is, $\mu$ is the unique Borel probability measure on $\R$ satisfying
\[
\mu = \sum_{j=1}^m p_j \mu \circ S_j^{-1}.
\]
(see \cite{Hutchinson1981}).  It is well known that $\mu$ is either absolutely continuous or purely singular with respect to the Lebesgue measure on $\R$.  However,  it remains a fundamental and open problem  to judge the type of  $\mu$  in the above general setting (see e.g. \cite{Solomyak1995, PeresSchlagSolomyak1998, Shmerkin2014, ShmerkinSolomyak2016} and the references therein). Below is a special restricted version of this problem.

\begin{ques}
\label{ques-2} Let $\mu$ be the self-similar measure generated by $\{S_j(x)=\rho x + b_j\}_{j=1}^m$ and a probability vector $\{p_j\}_{j=1}^m$.
Suppose that $\{S_j\}_{j=1}^m$ satisfies the finite type condition (see Section~\ref{S-5} for the definition).  Can we determine whether $\mu$ is absolutely continuous?
\end{ques}

There are many  examples of iterated function systems which allow overlaps but satisfy the finite type condition (see \cite{NgaiWang2001}).  In \cite[Theorem 1.3]{LauNgaiRao2001},  Lau, Ngai and Rao provided a confirmative answer to Question \ref{ques-2}. They proved that $\mu$ is absolutely continuous  if and only if
certain constructed matrix has spectral radius $\rho$. Alternatively, Protasov \cite{Protasov2000} provided an algorithm to check the absolute continuity of $\mu$ by the Fourier analysis approach, in the special case  when $\{S_j\}_{j=1}^m$ is an integral iterated function system on $\R$, i.e., $S_j(x)=\frac{1}{N}(x+d_j)$ with $N\geq 2$ being an integer and $d_j\in \Z$  (see Remark \ref{rem-6.2}).

As an analogue of Question \ref{ques-2},  the following  problem is on certain self-affine measures (see Section~\ref{S-6} for the definition).
\begin{ques}
\label{ques-2'}
Let $d\geq 2$. Let $\mu$ be the self-affine measure generated by a family of affine maps $\{S_j(x)=A^{-1}(x+d_j)\}_{j=1}^m$ on $\R^d$ and a probability vector $\{p_j\}_{j=1}^m$,   where
$A$ is a $d\times d$ expanding integer matrix and $d_j\in \Z^d$. Can we determine whether $\mu$ is absolutely continuous?
\end{ques}

In \cite{DengHeLau2008}, Deng, He and Lau investigated this question. They established a vector representation for $\mu$ via matrix products, and showed that
$\mu$ is absolutely continuous if and only if the corresponding matrix products have certain limiting behaviors.   However  there is no efficient algorithm to check these limiting behaviors directly (see  Remark \ref{rem-L}).  Alternatively, one can use the Fourier analysis approach to give an equivalent condition for $\mu$ to be absolutely continuous (see Proposition \ref{pro-self-affine}(iii)). Nevertheless, it is unlikely that Protasov's algorithm in \cite{Protasov2000} can be extended to check this condition (see Remark \ref{rem-6.2}).

Next we address the question on sofic affine-invariant sets on  the $2$-torus ${\Bbb T}^2=\R^2/\Z^2$. Let $m, n$ be positive integers with $n > m$. Let $T$ be the affine endomorphism on ${\Bbb T}^2$ represented by  the $2\times 2$  diagonal matrix $\mbox{diag}(n,m)$.  Write $$ D = \{0, \dots, n-1\}\times \{0, \dots, m-1\}.$$
Define a map $R_T : D^\N \to \mathbb{T}^2$  by
\[
R_T( (x_k, y_k)_{k=1}^\infty): = \sum_{k=1}^\infty \left(\begin{array}{cc}n^{-k} & 0\\
0& m^{-k}\end{array}\right)\left(\begin{array}{c}x_k\\
y_k\end{array} \right).
\]
Let $A = (a_{ij})_{i,j \in D}$ be a positively irreducible $0$-$1$ matrix. Then $A$ defines an irreducible subshift of finite type $\Sigma_A\subset D^\mathbb{N}$  by
\[
\Sigma_A: = \left\{ (z_k)_{k=1}^\infty :\; a_{z_k z_{k+1}} = 1 \mbox{ for } k\geq 1\right\}.
\]
Now let $K_T(A): = R_T(\Sigma_A)$. Then $K_T(A)$ is a $T$-invariant subset of ${\Bbb T}^2$. This is the model of  sofic affine-invariant sets studied in \cite{KenyonPeres1996, KenyonPeres1996b}, which is a  generalization of the  class of  Bedford-McMullen carpets (cf. \cite{Bedford1984, McMullen1984}). A natural and important  question  which arises here is that whether the Hausdorff dimension and the box-counting dimension of  $K_T(A)$ coincide. The reader is referred to \cite{Falconer2003, Mattila1995} for the definitions of these dimensions.

In \cite{KenyonPeres1996,KenyonPeres1996b},  Kenyon and Peres gave implicit formulas of the Hausdorff and box-counting dimensions of  $K_T(A)$ in terms of some dynamical notions (e.g. topological entropy, pressure, and measure-theoretic entropy). They showed that these two dimensions coincide if and only if
the unique invariant measure of maximal entropy on $\Sigma_A$ projects via $\pi$ to the invariant measure of maximal entropy on the sofic shift $\pi (\Sigma_A)$, where $\pi$ is the projection map given by $(x_k, y_k)_{k=1}^\infty\mapsto (y_k)_{k=1}^\infty$. It leads to the following.

\begin{ques}
\label{ques-3}
In the above setting, can one determine  whether the unique invariant measure of maximal entropy on $\Sigma_A$ projects via $\pi$ to the invariant measure of maximal entropy on the sofic shift $\pi (\Sigma_A)$?
\end{ques}

In \cite[p.~161]{KenyonPeres1996},  Kenyon and Peres mentioned that the answer of Question \ref{ques-3} is positive. However they did not give a detailed justification.

In this paper, we show that  Questions \ref{ques-2}-\ref{ques-3} can be reduced to Question \ref{ques-1} (see Theorems \ref{thm-6.1}, \ref{thm-6.4} and \ref{thm-5.1}, respectively). Indeed, for each of  Questions \ref{ques-2}-\ref{ques-3},  we can construct a tuple $\bM=(M_1,\ldots, M_k)$ of non-negative square matrices, so that ${\bf M}$ is positively irreducible and the question is reduced to determining whether $\bM$ has a  uniform Lyapunov exponent modulo $0$.

Furthermore, we show that the answer to Question \ref{ques-1} is positive. This is done by providing  checkable criteria under the assumptions of irreducibility and positive irreducibility, respectively. As a consequence, we are able to give affirmative answers to  Questions \ref{ques-2}-\ref{ques-3} using this new approach. Moreover, we can derive some new properties of the self-similar/self-affine  measures  considered in Questions \ref{ques-2}-\ref{ques-2'} (see Corollary \ref{cor-6.1}, Theorem \ref{thm-cts}). For instance, we show that if these measures are singular, then their Hausdorff dimensions are strictly less than the dimensions of ambient spaces. Moreover for the self-similar measure $\mu$ considered in Question \ref{ques-2},  we give a checkable criterion for deciding the absolute continuity of $\mu$ with respect to the  $s$-dimensional Hausdorff measure ${\mathcal H}^s|_K$ restricted on $K$, where $s=\dim_HK$,  and show that if $\mu$ is absolutely continuous with respect to the Lebesgue measure on $\R$ then, restricted on  certain open interval, the density function $\frac{d\mu}{dx}$ only takes values in $(c_1, c_2)$ for some positive constants $c_1$ and $c_2$.

To state our criteria for Question \ref{ques-1}, we first consider the non-negative case. Suppose that ${\bM}=(M_1,\ldots, M_k)$ is a  tuple of non-negative $d\times d$  matrices and  ${\bM}$ is positively irreducible. Set $\A=\{1,\ldots, k\}$  and write
\begin{equation}
\label{e-YM}
Y_{{\bf M}}:=\left\{ (j_n)_{n=1}^\infty \in \mathcal{A}^\N :\; M_{j_1 \cdots j_m} \neq {\bf 0} \mbox{ for all } m \geq 1\right\},
\end{equation}
where we adopt the convention that $M_{j_1 \cdots j_m}=M_{j_1}\cdots M_{j_n}$.  Then $Y_{\bf M}$ is an irreducible sofic shift over $\A$ (see Proposition \ref{Y is irred sofic}).   It is well known that the topological entropy of a sofic shift is computable (see Section~\ref{S-3.2}).   Write
\begin{equation}
\label{e-RM}
r({\bf M})=\exp(\log \rho(M_1+\cdots+M_k)-h_{\rm top}(Y_{\bf M})),
\end{equation}
where $\rho(A)$ stands for the spectral radius of $A$ (i.e. the maximal modulus of eigenvalues of $A$),   and  $h_{\rm top}(Y_{\bf M}$) denotes the topological entropy of  $Y_{\bf M}$.  Then $r({\bf M})$ is computable.

 Set
 \begin{equation}
 \label{e-1.4}
 {\mathcal J}:=\left\{j_1\cdots j_n\in \A^n:\; 1\leq n\leq d^2, \; (M_{j_1\cdots j_n})_{1,1}\neq 0\right\}.
 \end{equation}
Then  $ {\mathcal J}\neq \emptyset$ (see Lemma \ref{lem-irr}). Define a $d\times d$ matrix $B$  by
 \begin{equation}
 \label{e-B}
 B=\frac{1}{\#  ({\mathcal J})}\sum_{J\in {\mathcal J}}M_J,
 \end{equation}
where the symbol $\#$  stands for  the cardinality. The matrix $B$ might not be positively irreducible. Here we consider its irreducible decomposition. Indeed, there exists a permutation matrix $T$ such that $T^{-1}B T$ has the following block upper triangular form:
\begin{equation}
\label{e-blockB}
T^{-1}BT=\left(
\begin{array}{cccc}
B^{(1)} & * & \ldots & *\\
0 & B^{(2)} & * & \vdots\\
\vdots& & \ddots & * \\
0 & \ldots &  0& B^{(t)}
\end{array}
\right)
\end{equation}
with square diagonal blocks of sizes $d_i$, $i=1,\ldots, t$, $\sum_{i=1}^t d_i=d$, so that for each $i=1,\ldots, t$, either $B^{(i)}$ is positively irreducible or $B^{(i)}={\bf 0}$.

Set
\begin{equation}
\label{e-Lambda}\Lambda=\{i:\; 1 \leq i\leq t,\; B^{(i)}\neq {\bf 0}\}.
\end{equation}
For $i\in \Lambda$, let $v_i, u_i\in \R^{d_i}$ be the left and right positive eigenvectors of $B^{(i)}$  corresponding to the eigenvalue $\rho\left(B^{(i)}\right)$, satisfying $v_i^{\T} u_i=1$, where the superscript $\T$ stands for  transpose.  The existence of such eigenvectors is ensured by the Perron-Frobenius theory (see e.g. \cite[Theorem 8.4.4]{HornJohnson1985}).

For  $J\in {\mathcal J}$, partition $T^{-1}M_JT$ into the form
\begin{equation}
\label{e-block}
T^{-1}M_JT=\left(
\begin{array}{cccc}
M_J^{(1)} & * & \ldots & *\\
* & M_J^{(2)} & * & \vdots\\
\vdots& & \ddots & * \\
* & \ldots &  *& M_J^{(t)}
\end{array}
\right)
\end{equation}
with  block sizes the same as  in \eqref{e-blockB}. By the definition of $B$, $T^{-1}M_JT$ is also block upper triangular for $J\in {\mathcal J}$. Moreover,  this is true for all $J\in \A^*$ with $(M_J)_{1,1}>0$  (see Lemma \ref{lem-3.7}). For $J\in \J$, we let $|J|$ denote the length of $J$, i.e.~$|J|=n$ if $J=j_1\cdots j_n$.  Now we are ready to state one of our criteria for the non-negative case.
\begin{thm}
\label{thm-1.4}
Suppose that ${\bf M}$ is positively irreducible. Then ${\bf M}$ has a uniform Lyapunov exponent modulo $0$ if and only if there exists $i\in \Lambda$ such that
\begin{equation}
\label{e-cr}
v_i^\T M_J^{(i)}u_i=r({\bf M})^{|J|} \qquad \mbox{ for all }\;J\in {\mathcal J}.
\end{equation}
\end{thm}

Since $r({\bf M})$ is computable and ${\mathcal J}$ is a finite set, the above theorem provides an algorithm for deciding whether ${\bf M}$ has a uniform Lyapunov exponent modulo $0$.

Next we consider the general case that ${\bf M}$ consists of real $d\times d$ matrices.  For $q>0$,  define 
\begin{equation}
\label{e-pressure}
P({\bf M}, q)=\lim_{n\to \infty} \frac{1}{n}\log \sum_{i_1\cdots i_n\in \A^n} \|M_{i_1}\cdots M_{i_n}\|^q,\quad q>0.
\end{equation}
The existence of the above limit follows by subadditivity. We call $P({\bf M}, \cdot)$ the {\it pressure function} associated with ${\bf M}$. In  \cite{Zhou1998}, Zhou proved that  that  $P({\bf M},q)$ is computable for every even positive integer $q$; more precisely, $$P({\bf M},q)=\log \rho\left(\sum_{i=1}^kM_i^{\otimes q}\right)$$
 for even $q$, where $A^{\otimes q}=A\otimes\cdots\otimes A$ is the $q$-fold Kronecker product of $A$.

The following is another checkable criterion for Question \ref{ques-1}.

\begin{thm}
\label{thm-1.6}
Suppose that ${\bf M}$ is irreducible or positively irreducible. Then ${\bf M}$ has a uniform Lyapunov exponent modulo $0$ if and only if 
\begin{equation}
\label{e-PC}
 P({\bf M}, 2)+ P({\bf M}, 6)=2P({\bf M}, 4).
\end{equation}
\end{thm}

The above result is somehow unexpected since, for certain given tuple of general matrices,  it is even undecidable whether the zero matrix is in the semigroup generated by these matrices (see \cite{Paterson1970} and also \cite{BlondelTsitsiklis1997, Cassaigne2014}). This result might also have potential applications in detecting the existence of $L^1$-solutions for general refinement equations in wavelet theory.

 We remark that in the non-negative case, although the condition \eqref{e-PC} looks easier to check than \eqref{e-cr},  it provides less information in classifying those tuples having a uniform Lyapunov exponent modulo $0$.
 
 Next we address some related works in the literature.  Most related to the above results (Theorems \ref{thm-1.4}-\ref{thm-1.6}) are the recent works by Protasov and Voynov \cite{ProtasovVoynov2014} and Morris \cite{Morris2016}. In \cite{ProtasovVoynov2014}, Protasov and Voynov studied when a matrix semigroup has constant spectral radius,  in the sense that  the spectral radius of all its elements is the same and non-zero. Among other things, Protasov and Voynov  pointed out  that for an  irreducible or positively irreducible  tuple $\bM=(M_1,\ldots, M_k)$,  the multiplicative semigroup ${\mathcal S}(\bM)$  generated by $\bM$  has constant spectral radius  if and only if
\begin{equation}
\label{e-PV}
C^{-1}\leq \|M\|\leq C \qquad \mbox{ for some constant $C>0$ and  all }M\in {\mathcal S}(\bM).
\end{equation}
 This fact follows from \cite[Theorem~4.7]{OmladicRadjavi1997}  which says,  for any irreducible matrix semigroup ${\mathcal S}$ with constant spectral radius, there is a norm in $\R^d$ such that the induced operator norm of all matrices from ${\mathcal S}$ is $1$.    Moreover, in the case when ${\bf M}$ is positively irreducible,  Protasov and Voynov proved that if $A$ is  an irreducible matrix in the convex hull of ${\mathcal S}(\bM)$ with $\rho(A)=1$ and $v$ is the right Perron-Frobenius eigenvector of $A$, then \eqref{e-PV} holds if and only if all matrices in ${\mathcal S}(\bM)$
have a common invariant linear subspace that contains all vectors $v-Mv$, $M\in {\mathcal S}(\bM)$, and does not contain $v$.   Based on this criterion, they provided an efficient algorithm for deciding whether \eqref{e-PV} holds (see \cite[Section 7.1]{ProtasovVoynov2014}). In the general case when ${\bf M}$ is irreducible, Protasov and Voynov proved  \eqref{e-PV} holds if an only if $P({\bf M}, 2)=P({\bf M}, 4)=\log k$ (see \cite[Section~7.3]{ProtasovVoynov2014}). For some other studies on matrix semigroups with constant spectral radius or multiplicative spectral radius, one is referred to  \cite{OmladicRadjavi1997, Popov2013}.

In \cite[Theorem 10]{Morris2016}, among other things, Morris proved that for an irreducible tuple  $\bM$ of real matrices,  $P({\bf M}, q)$ is an affine function of $q$ on $(0,\infty)$ if and only if there exists $\lambda\in \R$ such that
\begin{equation}
\label{e-sp}
\rho(M_{i_1}\cdots M_{i_n})\in \{0, e^{\lambda n}\} 
\end{equation}
for any $n\in \N$ and $i_1,\ldots, i_n\in \{1,\ldots, k\}$.  It is easy to see that the property \eqref{e-ULE} implies \eqref{e-sp}. Hence by Morris' result, a necessary condition for the property \eqref{e-ULE} is the affinity of  $P({\bf M}, q)$  on $(0,\infty)$. 

In the remaining part of this section,  we outline the main steps in our proofs of Theorems \ref{thm-1.4}-\ref{thm-1.6}. First suppose that ${\bf M}=(M_1,\ldots, M_k)$ is positively irreducible. It is clear that ${\bf M}$  has a uniform Lyapunov exponent modulo $0$ if and only if that for any $c>0$,   $c {\bf M}:= (c M_1,\ldots, cM_k)$  has this property. Multiplying ${\bf M}$ by   the scalar $1/r({\bf M})$ if necessary, we may assume that
${\bf M}$ is normalized in the sense that $r({\bf M})=1$ (see Lemma \ref{lem-normalize}(i)).    Now it is easy to show that ${\bf M}$  has a uniform Lyapunov exponent modulo zero if and only if
\begin{equation}
\label{e-PV1}
C^{-1}\leq \|M\|\leq C \qquad \mbox{ for some constant $C>0$ and  all }M\in {\mathcal S}(\bM)\backslash \{{\bf 0}\}.
\end{equation}

Comparing this with \eqref{e-PV}, the main difference lying here is that  the zero matrix is allowed to be included in  ${\mathcal S}(\bM)$. Although the difference looks slight, it brings significant difficulties to the  study. To investigate when \eqref{e-PV1} holds, set
$$
{\mathcal U}:=\left\{j_1\cdots j_n\in \A^n:\; n\in \N,\; \left(M_{j_1\cdots j_n}\right)_{1,1}>0\right\}.
$$
Then the collection $\{M_J:\; J\in {\mathcal U}\}$  becomes a semigroup. Using the positive irreducibility assumption of  ${\bf M}$, we are able to show that  \eqref{e-PV1} holds if and only
\begin{equation}
\label{e-PV2}
C^{-1}\leq \|M_J\|\leq C \qquad \mbox{ for some constant $C>0$ and all }J\in {\mathcal U}.
\end{equation}
However, the semigroup $\{M_J:\; J\in {\mathcal U}\}$  might  not be positively irreducible.  For instance, this is the case when
$$
{\bf M}=\left\{
\left(\begin{array}{cc}
1 & 0\\
0 & 0
\end{array}
\right),
\left(\begin{array}{cc}
0 & 1\\
0 & 0
\end{array}
\right),
\left(\begin{array}{cc}
0 & 0\\
1 & 0
\end{array}
\right),
\left(\begin{array}{cc}
0 & 0\\
0 & 1
\end{array}
\right)
\right\}.
$$
  As a key part of our proof,   using symbolic dynamics and the thermodynamic formalism for matrix products,  we show that \eqref{e-PV2}
holds if and only if there exists $i\in \Lambda $ such that
\begin{equation}
\label{e-PV3}
C^{-1}\leq \|M_J^{(i)}\|\leq C \qquad \mbox{ for some constant $C>0$ and all }J\in {\mathcal U},
\end{equation}
where $M_J^{(i)}$ is the $i$-th diagonal block in the partitioned matrix $T^{-1}M_JT$ as in \eqref{e-block}.    For $i\in \Lambda$, since $B^{(i)}$ is positively irreducible and   $B^{(i)}$ lies in the convex hull of  $\{M_J^{(i)}:\; J\in {\mathcal U}\}$,  applying the Perron-Frobenius theory of non-negative matrices,   we are able to show that \eqref{e-PV3} holds if and only if $v_i^\T M_J^{(i)}u_i=1$  for all $J\in {\mathcal U}$. Finally, an additional argument shows that the latter condition is equivalent to  $v_i^\T M_J^{(i)}u_i=1$  for all $J\in {\mathcal J}$, from which Theorem \ref{thm-1.4} follows.

Next we outline the proof of Theorem \ref{thm-1.6}. Suppose that ${\bf M}$ is irreducible or positively irreducible. Applying  the thermodynamic formalism of matrix products, we are able to show that the following three properties are equivalent: (i) $P({\bf M}, q)$ is affine on $(0,\infty)$; (ii) $P({\bf M}, q)$ is affine on $(a,b)$ for some $0<a<b$; (iii) ${\bf M}$ has a uniform Lyapunov exponent modulo $0$. The proof of this part is somehow similar to the argument in \cite[Theorem 10]{Morris2016}. Since the pressure function $P({\bf M}, q)$ is always convex, the condition \eqref{e-PC} implies the affinity of $P({\bf M}, q)$ on the interval $[2,6]$, and hence implies that ${\bf M}$ has a uniform Lyapunov exponent modulo $0$.

The paper is organized as follows: In Section~\ref{S-2.1}, we give some notation and preliminaries
about symbolic dynamics and  the thermodynamic formalism for matrix products. In Section~\ref{S-3}, we give further properties of matrix products.  The proofs of Theorems \ref{thm-1.4}-\ref{thm-1.6} are given in Sections~\ref{S-4}-\ref{S-new}. In Sections~\ref{S-5}-\ref{S-7}, we consider Questions \ref{ques-2}-\ref{ques-3}  respectively. In Section~\ref{S-8}, we give some final remarks and questions.

\section{Notation and Preliminaries} \label{S-2.1}

In this section, we provide some necessary notation and preliminaries.  For
two families of real numbers $\{a_i\}_{i \in \I}$ and $\{b_i\}_{i \in \I}$, we write
\begin{align*}
a_i \approx b_i \ \ & \text{ if there is } c>0 \text{ such that } c^{-1} b_i \le a_i \le c b_i\text{ for all } i \in \I; \\
a_i \gtsim  b_i \ \ & \text{ if there is } c>0 \text{ such that }  a_i \ge c b_i \text{ for all } i \in \I; \\
a_i \ltsim  b_i \ \ & \text{ if there is } c>0 \text{ such that }  a_i \le c b_i \text{ for all } i \in \I.
\end{align*}

\subsection{Subshifts}
\label{S-3.1}
%
%

In this subsection, we introduce some basic notation and definitions about subshifts. The reader is referred to \cite{LindMarcus1995} for the background and more details.

Let $\A$ be a finite set of symbols which will be called the \emph{alphabet}.  Let
\[
\A^* = \bigcup_{k=0}^\infty \A^k
\]
denote the set of all finite words with letters from $\A$, including the empty word $\varepsilon$.
Denote the length of a word $I$ by $|I|$, that is, $|I| = k$ if $I \in \A^k$.
Let
\[
\A^\N = \{ (x_i)_{i=1}^\infty :\; x_i \in \A \textnormal { for } i \geq 1\}
\]
denote the set of all infinite sequences of elements from $\A$.  Then $\A^\N$ is a compact metric space under the product topology, which can be induced by the metric
\[
d(x, y) = 2^{-\inf\{k :\; x_k \neq y_k\}}, \ \  \mbox{ for } x = (x_i)_{i=1}^\infty,\; y = (y_i)_{i=1}^\infty.
\]
For $n \in \N$ and $I \in\A^n$, set
\begin{equation}
[I] = \left \{ (x_i)_{i=1}^\infty \in \A^\N :\; x_1 \cdots x_n = I\right\}
\label{cylinder set definition}
\end{equation}
and call it an {\it $n$-th cylinder set} in $\A^\N$.

Define the shift transformation $\sigma : \A^\N \to \A^\N$ by $ (\sigma x)_i = x_{i+1}$ for all $i \in \N$.  Then $\sigma $ is a continuous self-map.  The pair $(\A^\N, \sigma)$ is a  topological dynamical system
and is called the \emph{one-sided full shift over $\A$}.

If $X$ is a compact $\sigma$-invariant subset of $\A^\N$, then the topological dynamical system $(X, \sigma)$ is called a \emph{one-sided subshift over $\A$}, or simply, a \emph{subshift}. Sometimes we write $(X,\sigma_X)$ instead of $(X,\sigma)$.

A word $I\in \A^*$ is said to be \emph{admissible} in a subshift $X$ if it occurs as a consecutive string in a sequence in $X$, that is, $[I]\cap X\neq \emptyset$.    Note that the empty word $\varepsilon$ is also admissible.  The \emph{language} $\L(X)$ of $X$ is the set of all admissible words in $X$, that is,
\[
\L(X) = \{ I \in \A^*:\; I = x_1 \cdots x_n \mbox{ for some } x = (x_i)_{i=1}^\infty \in X \mbox{ and } n \ge 1\} \cup \{ \varepsilon\}.
\]
For $n \ge 0$, denote
\[
\L_n(X) = \{ I \in \L(X) :\; |I| = n\}.
\]

A subshift $X$ over $\A$ is said to be a \emph{subshift of finite type} if there is a matrix $A = (A_{\alpha, \beta})_{\alpha, \beta \in \A}$ with entries 0 or 1 such that
\[
X = \{ (x_i)_{i=1}^\infty  \in \A^\N :\; A_{x_i , x_{i+1} } = 1 \text{ for all } i\geq 1\}.
\]
If the matrix $A$ is positively irreducible (that is, for any $\alpha, \beta \in \A$, there is $N >0$ such that  $(A^N)_{\alpha, \beta}>0$), $X$ is called an \emph{irreducible subshift of finite type}.
Very often we use $\Sigma_A$ instead of $X$ to denote the above subshift of finite type.

Let $(X, \sigma_X)$ and $(Y, \sigma_Y)$ be two subshifts over finite alphabets $\A$ and $\A'$, respectively.  A continuous surjective map $\pi : X \to Y$ such that $\pi \circ \sigma_X = \sigma_Y \circ \pi$ is called a \emph{factor map}.  In this case $Y$ is said to be a \emph{factor} of $X$.

A  subshift $Y$ is called to be a \emph{sofic shift} if $Y$ is a factor of a subshift of finite type, say $X$.  If further $X$ is irreducible, then $Y$ is called an \emph{irreducible sofic shift}.

\subsection{Entropies and  Parry measures}
\label{S-3.2}
Let $(X, \sigma_X)$ be a subshift over a finite alphabet $\A$.  Denote by $\M(X)$ the set of all Borel probability measures on $X$.  Endow $\M(X)$ with the weak-star topology.  Denote by $\M(X, \sigma_X)$ the set of all $\sigma_X$-invariant Borel probability measures on $X$.  The sets $\M(X)$ and $\M(X, \sigma_X)$ are both non-empty, compact and convex (see e.g. \cite{Walters1982}). An element $\mu \in \M(X, \sigma_X)$ is called {\it ergodic} if $\mu(A)=1$ or $0$ for any  Borel set $A\subset X$ with $\sigma_XA\subset A$.

 Let $\L(X)$ and $\mathcal L_n(X)$ be defined as in the preceding subsection.
For convenience, for $\mu \in \M(X)$ and $I \in \L(X)$, we write
\[
\mu (I) : = \mu( [I] \cap X),
\]
where $[I]$ denotes a cylinder set in $\A^\N$ defined as in \eqref{cylinder set definition}.

Given $\mu \in \M(X, \sigma_X)$, the \emph{measure-theoretic entropy} of $\mu$ with respect to $\sigma_X$ is defined by
\begin{equation}
h_{\mu}(\sigma_X) : = - \lim_{n\to \infty}\frac{1}{n} \sum_{I \in \L_n(X)} \mu(I) \log \mu(I).
\end{equation}
The existence of the above limit  follows by  a standard sub-additivity argument.

The \emph{topological entropy} of $X$ with respect to $\sigma_X$ is defined as
\begin{equation}
h_{\mathrm{top}}(X) = \lim_{n\to \infty} \frac{1}{n} \log \#  (\L_n(X)),
\end{equation}
where $\#$ stands for  cardinality. Again, the above limit  exists by sub-additivity.

It is well known (cf. \cite[Chapter 8.3]{Walters1982}) that for any subshift $X$,
$$
h_{\mathrm{top}}(X)=\sup_{\mu\in \M(X, \sigma_X)} h_\mu(\sigma_X),
$$
and the supremum is attainable. Each $\mu\in \M(X, \sigma_X)$ so that $h_\mu(\sigma_X)=h_{\mathrm{top}}(X)$ is called  an {\it invariant measure of maximal entropy}.

The topological entropy of a subshift of finite type or  sofic shift is computable.  More precisely, if $X=\Sigma_A$ is a subshift of finite type associated with a $0$-$1$ matrix $A$, then $h_{\mathrm{top}}(X)=\log \rho(A)$; and if $X$ is a sofic shift, then $h_{\mathrm{top}}(X)=\log \rho(A_G)$, where $A_G$ is the incidence matrix of a right-resolving graph presentation of $X$. For details, see \cite[Chapter 4]{LindMarcus1995}.

The following result is due to Parry. The reader is referred to \cite[Theorem 5.5]{Feng2011} for certain generalization and a detailed proof.
\begin{thm}[\cite{Parry1964}]
\label{thm-Parry}
Suppose that  $(X,\sigma_X)$ is an irreducible subshift of finite type, or an irreducible sofic shift over a finite alphabet. Then
\begin{equation}
\label{e-p1} \# ({\mathcal L}_n(X))\approx e^{n h_{\rm top}(X)}\quad \mbox{ for } n\in \N.
\end{equation}
Moreover $\sigma_X$ has a unique invariant measure of maximal entropy, say $\nu$. Furthermore, $\nu$ is ergodic and it is the unique invariant measure  satisfying  the following property:
\begin{equation}
\label{e-p2}
\nu(I)\approx  e^{-n h_{\rm top}(X)} \quad \mbox{ for }n\in \N,\; I\in \mathcal L_n(X).
\end{equation}
\end{thm}

The measure $\nu$ in the above theorem is called the {\it Parry measure} on $X$.

\subsection{Lyapunov exponents and  the thermodynamic formalism for matrix products}
\label{S-pressure}
Let ${\bf M}=(M_1,\ldots, M_k)$ be a tuple of real $d\times d$ matrices. Write $\A=\{1,\ldots, k\}$. Fix a matrix norm $\|\cdot\|$  on $\R^{d\times d}$ by
$\|A\|=\sum_{1\leq i,j\leq d}|a_{i,j}|$ for $A=(a_{i,j})$. The following result follows from Kingman's sub-additive ergodic theorem.
\begin{thm}[{\cite[Theorem 10.1]{Walters1982}}]
\label{thm-FK}
For any ergodic measure $\mu$ on $\A^\N$, one has
$$
\lim_{n\to \infty} \frac{1}{n}\log \|M_{x_1}\cdots M_{x_n}\|=\lambda({\bf M}, \mu)\quad \mbox{ for $\mu$-a.e.~$x=(x_n)_{n=1}^\infty$},
$$
where $$\lambda({\bf M}, \mu)=\lim_{n\to \infty}\frac{1}{n} \sum_{i_i\cdots i_n\in \A^n} \mu([i_1\cdots i_n])\log \|M_{i_1}\cdots M_{i_n}\|.$$
\end{thm}

We call  $\lambda({\bf M}, \mu)$ the {\it Lyapunov exponent} of ${\bf M}$ with respect to $\mu$.

Recall that the pressure function $P({\bf M},q)$ is defined as in \eqref{e-pressure}.  The following result is a corollary of the sub-additive variational principle established in \cite{CaoFengHuang2008} (for  earlier results in the non-negative or invertible case, see \cite{Feng2004, Kaenmaki2004}).
\begin{thm}
For any $q>0$, we have $$P({\bf M}, q)=\sup\{h_\mu(\sigma)+q\lambda({\bf M}, \mu):\; \mu\in {\mathcal M}(\A^\N, \sigma)\}.$$
\end{thm}
We say that $\mu$ is an {\it equilibrium state} for $({\bf M}, q)$ if it attains the above supremum. 

The following result describes the Gibbs property of matrix equilibrium states. 

\begin{thm}[\cite{FengKaenmaki2011, FengLau2002}]
\label{thm-matrix} 
Suppose that ${\bf M}$ is  irreducible or positively irreducible. Let $q>0$. There exists a unique $\nu=\nu_q\in {\mathcal M}(\A^\N, \sigma)$ such that
 $$
 \nu([i_1\cdots i_n])\approx \exp(-nP({\bf M}, q))\|M_{i_1}\cdots M_{i_n}\|^q\quad \mbox{ for }n\in \N, \; i_1\cdots i_n\in \A^n.
 $$
Moreover, $\nu$ is ergodic and it is the unique equilibrium state for $({\bf M}, q)$. 
\end{thm}

\subsection{Irreducible decompositions}
Let ${\bf M}=(M_1,\ldots, M_k)$ be a tuple of non-negative $d\times d$ matrices. Suppose that ${\bf M}$ is non-trivial in the sense that for each $n\in \N$ there exists $i_1\cdots i_n\in \{1,\ldots, k\}^n$
such that $M_{i_1}\cdots M_{i_n}\neq 0$.   It is possible that ${\bf M}$ is not positively irreducible. In such situation, it is an elementary fact (see e.g. \cite[Proposition 1.4]{FengKaenmaki2011}) that one can always find a permutation matrix $T$, $t\in \{1,\ldots,d\}$ and positive integers $d_1,\ldots, d_t$ with $d_1+\cdots+d_t=d$ such that for each $j\in \{1,\ldots, k\}$, $T^{-1}M_jT$ has the following block upper triangular form:
\begin{equation}
T^{-1}M_jT=\left(
\begin{array}{cccc}
M_j^{(1)} & * & \ldots & *\\
0 & M_j^{(2)} & * & \vdots\\
\vdots& & \ddots & * \\
0 & \ldots &  0& M_j^{(t)}
\end{array}
\right)
\end{equation}
with square diagonal blocks of sizes $d_i$, $i=1,\ldots, t$; moreover, for each $i=1,\ldots, t$, the tuple ${\bf M}^{(i)}:=\left(M_1^{(i)},\ldots, M_k^{(i)}\right)$ is either positively irreducible,
or  consisting only of zero matrices $\bf 0$.

Let $\Gamma:=\{1\leq i\leq t: {\bf M}^{(i)} \mbox{ is positively irreducible}\}$. The following property plays a key role in our proof of Theorem \ref{thm-1.4}.

\begin{pro}
[{\cite[Proposition 1.4]{FengKaenmaki2011}}]
\label{pro-FK}
For any ergodic measure $\mu$ on $\A^\N$, we have
$$
\lambda({\bf M},\mu)=\max_{i\in \Gamma}\lambda\left({\bf M}^{(i)},\mu\right),
$$
where $\lambda({\bf M},\mu)$ is the  Lyapunov exponent
of ${\bf M}$ with respect to $\mu$ (see Section~\ref{S-pressure}).
\end{pro}

We remark that  the above proposition was only proved in \cite{FengKaenmaki2011} for  different irreducible decompositions. But the proof therein works well in our new setting.

\section{Irreducible tuples of non-negative matrices}
\label{S-3}

Throughout this section, let ${\bf M}=(M_1, \dots, M_k)$ be a tuple of non-negative $d \times d$ matrices, and suppose that ${\bf M}$ is positively irreducible.
We give several properties of ${\bf M}$,  some of which will be needed in the proof of Theorem \ref{thm-1.4}.

We begin with a simple fact.

\begin{lem}[{\cite[Lemma 8.4.1]{HornJohnson1985}}]
\label{lem-irr}
$\sum_{\ell=1}^d(M_1+\cdots+M_k)^\ell$ is a positive matrix.
\end{lem}
  Set $\A = \{1, \ldots, k \}$ and let $Y_{\bf M}$ be defined as in \eqref{e-YM}.

\begin{pro} \label{Y is irred sofic}
$Y_{{\bf M}}$ is an irreducible  sofic shift over $\A$. Moreover,
$$
\mathcal L_n(Y_{{\bf M}})=\{J\in \A^n:\; M_J\neq {\bf 0}\},\quad n\in \N,
$$
where $\mathcal L_n(Y_{{\bf M}})$ stands for the collection of admissible words of length $n$ in $Y_{{\bf M}}$ (see Section~\ref{S-3.1}).
\end{pro}

\begin{proof}
The result is most likely known, but we have not been able to find a reference so a proof is given for the reader's convenience.

Set $\D = \{1, \ldots, d\}$. Construct a subset $\mathcal{F}$ of $\D \times \mathcal{A}$ by
\[
\mathcal{F} =\left \{ (i,j)\in \D \times \mathcal{A}:\; \text{there exists } l \in \D \text{ such that } (M_j)_{i,l}>0\right\}.
\]
Define a $0$-$1$  matrix $A=(A_{u,v})_{u, v\in \mathcal F}$ by
\[
A_{(i,j),(i',j')} = \begin{cases}
1 & \text{ if }\; (M_j)_{i, i'} >0, \\0 & \text{ otherwise. }
\end{cases}
\]

Let $\Sigma_A$ be the subshift of finite type over ${\mathcal F}$ associated with $A$. We first show that $\Sigma_A$ is irreducible.  Fix $(i,j), (i',j') \in \mathcal{F}$.  By definition, $(M_j)_{i,i_1} >0$ for some $i_1 \in \D$.
Since ${\bf M}$ is positively irreducible, there exist $n\in \N$ and  $j_1\cdots j_n \in \mathcal{A}^n$ such that
$(M_{j_1\cdots j_n})_{i_1, i'} >0$. Therefore we can find $i_2, \dots, i_n \in \D$ such that
\[
(M_{j_1})_{i_1,i_2}\cdots (M_{j_{n-1}})_{i_{n-1}, i_n} (M_{j_n})_{i_n, i'} >0.
\]
Hence the word $(i,j)(i_1,j_1)\cdots (i_n, j_n)(i',j')$ is $A$-admissible. Therefore $\Sigma_A$ is irreducible.

Notice that
 $$
 (i_1,j_1)\cdots (i_n, j_n) \in \L(\Sigma_A) \Longleftrightarrow  (M_{j_1})_{i_1, i_2} \cdots (M_{j_{n-1}})_{i_{n-1}, i_n} >0.
 $$
It follows that
\begin{equation}
\label{e-equiv}
M_{j_1\ldots j_n}\neq  0 \Longleftrightarrow (i_1,j_1)\cdots (i_n, j_n) \in \L(\Sigma_A) \mbox{ for some }i_1,\ldots, i_n\in \D.
\end{equation}

Define  $\tau : \mathcal{F} \to \mathcal{A}$ by $(i,j)\mapsto j$.  Extend $\tau $ to a map $\pi: \mathcal{F}^\N\to \mathcal{A}^\N$ by
$$
\pi\left((x_n)_{n=1}^\infty\right)=\left(\tau(x_n)\right)_{n=1}^\infty.
$$
 By \eqref{e-equiv}, we have $Y_{{\bf M}}=\pi(\Sigma_A)$. Clearly, $\pi$ is a factor map.  Hence $Y_{{\bf M}}$ is an irreducible sofic shift.
 By \eqref{e-equiv}, we also have $\mathcal L_n(Y_{{\bf M}})=\{J\in \A^n:\; M_J\neq {\bf 0}\}$ for $n\in \N$.
 \end{proof}

Recall that the pressure function $P({\bf M}, \cdot)$ is defined as in \eqref{e-pressure}. Write  \begin{equation}
\label{e-pM}
P({\bf M}):=P({\bf M}, 1)
\end{equation}
 and call it  the {\it topological pressure} of ${\bf M}$. 
 
\begin{lem}
$P({\bf M})=\log \rho(M_1+\cdots+ M_k)$.
\end{lem} 
\begin{proof} Since $M_i$ are non-negative, we have
 $$
 \sum_{i_1\cdots i_n\in \A^n} \|M_{i_1}\cdots M_{i_n}\|= \left\|\sum_{i_1\cdots i_n\in \A^n} M_{i_1}\cdots M_{i_n}\right\|=\|(M_1+\cdots +M_k)^n\|. 
 $$
Now the lemma follows from the definition of $P({\bf M})$ and  Gelfand's Formula.
\end{proof}

Let $r({\bf M})$ be defined as in \eqref{e-RM}. 
\begin{de}
 We say that ${\bf M}$ is normalized if $r({\bf M})=1$.
\end{de}

\begin{rem}
\label{rem-1}
Let $a>0$. Then $P(a{\bf M})=P({\bf M})+\log a$ and  $Y_{a{\bf M}}=Y_{\bf M}$. Hence $r(a{\bf M})=ar({\bf M})$.
\end{rem}

\begin{lem}
\label{lem-normalize}
\begin{itemize}
\item[(i)] $\frac{1}{r({\bf M})}{\bf M}$ is normalized.
\item[(ii)]  ${\bf M}$ has a uniform Lyapunov exponent modulo $0$ if and only if
$$
\|M_J\|\approx (r({\bf M}))^{|J|} \quad \mbox { for }J\in \mathcal L(Y_{\bf M}).
$$
\end{itemize}
\end{lem}
\begin{proof}
Property (i) follows from Remark \ref{rem-1}. Next we prove (ii). By the definition of $Y_{\bf M}$, we see that  ${\bf M}$ has a uniform Lyapunov exponent modulo $0$ if and only if there exists a constant $\lambda\in \R$ such that
\begin{equation}
\label{e-MJ}
\|M_J\|\approx \exp(\lambda |J|) \quad \mbox { for }J\in \mathcal L(Y_{\bf M}).
\end{equation}
To show (ii), it suffices to show that
\begin{equation}
\label{e-lambda'}
\lambda=\log r({\bf M})=P({\bf M})-h_{\rm top}(Y_{\bf M})
\end{equation} when \eqref{e-MJ} holds.

Now suppose  \eqref{e-MJ} holds. Then
\begin{eqnarray*}
 \sum_{i_1\cdots i_n\in \A^n} \|M_{i_1\cdots i_n}\|&=&\sum_{i_1\cdots i_n\in \mathcal L_n(Y)} \|M_{i_1 \cdots i_n}\|\\
 &\approx & e^{\lambda n} \#  (\mathcal L_n(Y)).
\end{eqnarray*}
Hence by definition,  $P({\bf M})=\lambda+\lim_{n\to \infty} (1/n)\log \#  (\mathcal L_n(Y))=\lambda+h_{\rm top}(Y_{\bf M})$, and \eqref{e-lambda'} holds.
\end{proof}

\begin{pro}\label{lower bound implies upper bound}
Suppose furthermore  that  ${\bf M}$ is normalized. Then the following three statements are equivalent.
\begin{enumerate}
\item
$\| M_J \| \gtsim 1$  for  $J \in \L(Y_{\bf M})$.
\item
$\| M_J \| \ltsim 1$ for  $J \in \L(Y_{\bf M})$.
\item
$\| M_J \| \approx 1$ for  $J \in \L(Y_{\bf M})$.
\end{enumerate}
\end{pro}
\begin{proof}
It suffices to show that (1) is equivalent to (2).  By Proposition \ref{Y is irred sofic},  $Y_{\bf M}$ is an irreducible sofic shift over $\A$. Let $\nu$ denote the Parry measure on
 $Y_{\bf M}$ and $\mu$ the equilibrium measure for  $({\bf M},1)$.  By Theorems \ref{thm-Parry}-\ref{thm-matrix}, we have
 \begin{equation}
 \label{e-app}
 \begin{split}
 \nu([J]) & \approx \exp(-|J|h_{\rm top}(Y_{\bf M})), \\
  \mu([J])& \approx \|M_J\|\exp(-|J|P({\bf M}))
 \end{split}
 \end{equation}
 for $J\in \mathcal L(Y_{\bf M})$.
 Since  ${\bf M}$ is normalized, we have $h_{\rm top}(Y_{\bf M})= P({\bf M})$. Thus by \eqref{e-app}, we have
\begin{equation}
 \label{e-app1}
 \mu([J])\approx  \|M_J\|\cdot  \nu([J]) \quad \mbox{ for } J\in \mathcal L(Y_{\bf M}).
 \end{equation}
Below we show that (1) is equivalent to (2).

In one direction, if  (1) holds, then $\mu([J])\gtsim \nu([J])$ for $J\in \mathcal L(Y_{\bf M})$ by \eqref{e-app1},  which implies
$\nu\ll \mu$, and so $\nu=\mu$. Here we use the  fact that any two distinct ergodic measures on $Y_{\bf M}$ are mutually singular  (see, e.g. \cite[Theorem 6.10]{Walters1982}).
This together with \eqref{e-app1} yields $\|M_J\|\approx 1$ for  $J\in \mathcal L(Y_{\bf M})$. Hence (2) holds.

In the other direction, if (2) holds, then $\mu([J])\ltsim \nu([J])$ for $J\in \mathcal L(Y_{\bf M})$ by \eqref{e-app1}, which implies
$\mu\ll \nu$, and so $\nu=\mu$. Again we have $\|M_J\|\approx 1$ for  $J\in \mathcal L(Y_{\bf M})$. Hence (1) holds. This completes the proof.
\end{proof}

In the end, let $B$ be defined as in \eqref{e-B} and let $T$ be a permutation matrix so that $T^{-1}BT$ is a block upper triangular matrix of the form  in \eqref{e-blockB}, and for each $1\leq i\leq t$, either $B^{(i)}$ is positively irreducible or $B^{(i)}={\bf 0}$. Then we have the following result.
\begin{lem}
\label{lem-3.7}
For any $J\in \A^*\backslash\{\varepsilon\}$ with $(M_J)_{1,1}>0$, $T^{-1}M_JT$ is also a  block upper triangular matrix with the same block sizes as in \eqref{e-B}.
Moreover, for each $1 \le i \le t$, $B^{(i)} = {\bf 0}$ if and only if $M_J^{(i)} = {\bf 0}$ for all $J \in \A^*\setminus\{\varepsilon\}$ with $(M_J)_{1,1}>0$.
\end{lem}
\begin{proof}
It is enough to show that if an entry $B_{i,j}$ of $B$ is zero, then $(M_J)_{i,j}=0$ for every $J\in \A^*$ with $(M_J)_{1,1}>0$. To prove the result, suppose that $B_{i,j}=0$ for some $(i,j)\in \{1,\ldots,d\}^2$. Then by the definition of $B$, we have $(i,j)\neq (1,1)$ and
\begin{equation}
\label{e-J}
(M_J)_{i,j}=0 \quad \mbox{ for all }  J\in \A^* \mbox{ with }(M_J)_{1,1}>0 \mbox{ and }|J|\leq d^2.
\end{equation}

Suppose on the contrary that $(M_U)_{i,j}>0$ for some $U\in \A^*$ with $(M_U)_{1,1}>0$.  We may assume that $U$ is such word with minimal length.  By \eqref{e-J}, $|U|>d^2$ .   Write $U=u_1\cdots u_n$ with $n=|U|$. Since $(M_U)_{i,j}>0$ and $(M_U)_{1,1}>0$, there exist two words $i_1\cdots i_{n+1}$ and $j_1\cdots j_{n+1}$ over $\{1,\ldots, d\}$ such that
$$
i_1=i,\; i_{n+1}=j,\; j_1=1,\; j_{n+1}=1
$$
and
$$
(M_{u_s})_{i_s, i_{s+1}}>0,\quad (M_{u_s})_{j_s, j_{s+1}}>0\quad \mbox{ for   }s=1,\ldots, n.
$$

Since $n>d^2$, by the pigeon-hole principle, there exist $1\leq m<m'\leq n$ such that $(i_m, j_m)=(i_{m'}, j_{m'})$. Now set $U^*=u_1\cdots u_{m-1}u_{m'}\cdots u_{n}$. That is,
$U^*$ is obtained from $U$ by dropping off the sub-word $u_{m}\cdots u_{m'-1}$.   It is direct to see that  $(M_{U^*})_{i,j}>0$ and $(M_{U^*})_{1,1}>0$, which contradicts the minimality of the length of $U$.
\end{proof}

In  the end of this section, we present the following lemma which was pointed out to us by Wen Huang \cite{Wen2015}.
\begin{lem}
\label{lem-4.6} Let ${\bf A}=(A_1,\ldots, A_k)$ be a tuple of  $d\times d$  matrices, and let  $\mu$ be a fully supported ergodic measure on $\A^\N$ with $\A=\{1,\ldots, k\}$. Assume that
$\lambda({\bf A}, \mu)=0$, where $\lambda({\bf A}, \mu)$ is the  Lyapunov exponent of ${\bf A}$ with respect to $\mu$ (cf. Section~\ref{S-pressure}). Assume furthermore that there exists a constant $C>0$ so that
$$
\|A_{J}\|\leq C \quad  \mbox{ for all } J\in  \bigcup_{n=1}^\infty \A^n.
$$
Then we have
$$
\|A_{J}\|\geq C^{-1}\quad \mbox{ for all } J\in  \bigcup_{n=1}^\infty \A^n.
$$
\end{lem}
\begin{proof}
Suppose on the contrary that $\|A_{J}\|<C^{-1}$ for some finite word $J=j_1\cdots j_m\in \A^m$.  Then $$\gamma:=C\|A_J\|\in (0,1).$$  Below we derive a contradiction.

By the Birkhoff ergodic theorem (cf. \cite[Theorem 1.14]{Walters1982}), there exists a Borel set $F\subset \A^\N$ with $\mu(F)=1$ such that for all $x\in F$,
\begin{equation}
\label{e-Birk}
\lim_{n\to \infty}\frac{1}{n} \sum_{p=1}^n \chi_{[J]}(\sigma^p x)=\mu([J])>0,
\end{equation}
where $\chi_{[J]}$ denotes the characteristic function on $[J]$, and the last inequality follows from the assumption that $\mu$ is fully supported on $\A^\N$.

For $x\in F$, let $n_1(x)<n_2(x)<\cdots$ be all the positive integers $n$ so that $\sigma^n(x)\in [J]$, then we have $\lim_{j\to \infty} j/n_j(x)=\mu([J])$  by \eqref{e-Birk}.

Fix $x\in F$ and let $N_j=n_{(m+1)j}(x)$ for $j\geq 1$. Then $N_{j+1}-N_j\geq m+1$ and $$\lim_{j\to \infty} \frac{j}{N_j}=\frac{\mu([J])}{m+1}.$$
Observe that $x$ can be expressed as
$$
x=W_1JW_2J\cdots W_nJ\cdots
$$
with $W_1=x_1\cdots x_{N_1}$ and $W_n=x_{N_{n-1}+m+1}\cdots x_{N_{n}}$ for $n\geq 2$. Notice that
$$
\|A_{W_1JW_2J\cdots W_nJ}\|\leq \prod_{j=1}^n(\|A_{W_n}\|\cdot\|A_J\|)\leq \prod_{j=1}^n(C\cdot\gamma C^{-1})=\gamma^n,
$$
which implies
\begin{eqnarray*}
\liminf_{n\to \infty} \frac{1}{n}\log\|A_{x_1\cdots x_n}\|&\leq & \liminf_{n\to \infty} \frac{1}{N_{n}+m}\log  \|A_{W_1JW_2J\cdots W_nJ}\|\\
&\leq &  \liminf_{n\to \infty} \frac{n\log \gamma}{N_{n}+m}=\frac{\mu([J]) \log \gamma}{m+1}<0.
\end{eqnarray*}
This leads to a contradiction, since  by Theorem \ref{thm-FK} $$\lim_{n\to \infty} \frac{1}{n}\log\|A_{y_1\cdots y_n}\|=\lambda({\bf A},\mu)=0$$ for $\mu$-a.e.~$y\in \A^\N$.
\end{proof}

\section{Proof of Theorem \ref{thm-1.4}}
\label{S-4}
In this section, we prove Theorem \ref{thm-1.4}.  Suppose that ${\bf M}$ is positively irreducible.  Multiplying ${\bf M}$ by the scalar $1/r({\bf M})$  if necessary, we may assume that ${\bf M}$ is normalized, i.e., $r({\bf M})=1$.
Recall that
$$
\U=\{J\in \A^*:\; (M_{J})_{1,1}\neq 0\}.
$$

We first give two lemmas.
\begin{lem}
\label{lem-5.0}
There exists a constant $C>0$ such that for any $J\in {\mathcal L}(Y_{\bf M})$,
there exist $I_1, I_2\in \mathcal L(Y_{\bf M})$  satisfying that  $I_1JI_2\in \U$ and
\[
C^{-1}\|M_J\|\leq  \|M_{I_1JI_2} \| \leq  C \|M_{J}\|.
\]
\end{lem}

\begin{proof}
Since  ${\bf M}$ is positively irreducible, for each pair $(i,j)$ with $i, j\in \{1, \ldots, d\}$, we can choose a finite word $W(i,j) \in  \L(Y_{\bf M})$ such that
\[
(M_{W(i,j)})_{i,j} > 0.
\]
Fix these words $W(i,j)$ and   set $$c_1=\min_{1\leq i,j\leq d} (M_{W(i,j)})_{i,j},\quad c_2=\max_{1\leq i,j\leq d} \|M_{W(i,j)}\|.$$
Clearly $c_1,\; c_2>0$.

Now let $J \in \L(Y_{\bf M})$.  Then there exist $i,j\in \{1, \dots, d \}$ such that
\[
(M_J)_{i, j} \ge \frac{1}{d^2} \| M_J \|.
\]
Set $I_1=W(1,i)$ and $I_2=W(j,1)$. Then
\[
(M_{I_1JI_2})_{1, 1}\geq (M_{I_1})_{1,i} (M_J)_{i,j} (M_{I_2})_{j,1}\geq  \frac{c_1^2}{d^2} \| M_J \|,
\]
which implies $I_1JI_2\in \U$ and
\[
\frac{c_1^2}{d^2} \| M_J \|\leq \|M_{I_1JI_2}\|\leq \|M_{I_1}\|\|M_{I_2}\|\|M_J\|\leq c_2^2 \| M_J\|.
\]
This completes the proof of the lemma.
\end{proof}

\begin{lem}
\label{lem-5.2} Let $\mathcal S$ be a multiplicative semigroup  of non-negative $d\times d$ matrices satisfying
$$
\|A\|\approx 1 \quad \mbox { for }A\in \mathcal S.
$$
Then $$\|A\|\approx 1 \quad \mbox{ for }A\in \overline{\rm co}(\mathcal S),$$
where $\overline{{\rm co}}(\mathcal S)$ stands for the closure of the convex hull ${\rm co}(\mathcal S)$ of $\mathcal S$, recalling that
$${\rm co}(\mathcal S)=\left\{\sum_{i=1}^np_iA_i:\; n\in \N, \; p_i>0,\; A_i\in {\mathcal S} \mbox{ and } \sum_{i=1}^np_i=1\right\}.$$
\end{lem}
\begin{proof}
It follows from the simple fact that $\|\sum_{i=1}^np_iA_i\|=\sum_{i=1}^np_i\|A_i\|$.
\end{proof}

For  $A\subset \R^d$, let ${\rm aff}(A)$ denote the smallest affine subset of $\R^d$  containing $A$. This set is called the {\it affine hull} of $A$.
It is well known (cf. \cite[p.~6]{Rockafellar1970}) that
\begin{equation}
\label{e-aff}
{\rm aff}(A)=\left\{\sum_{i=1}^na_ix_i:\; n\in \N,\; a_i\in \R, \; x_i\in A \mbox{ and }\sum_{i=1}^n a_i=1\right\}.
\end{equation}

Let ${\mathcal J}$, $\Lambda$ be defined as in \eqref{e-1.4} and \eqref{e-Lambda}, respectively.
Recall that, for each $i \in \Lambda$, $v_i$, $u_i$ are the left and right positive eigenvectors of $B^{(i)}$ corresponding to the eigenvalue $\rho\left(B^{(i)}\right)$, respectively, satisfying $v_i^\T u_i=1$.

\begin{pro}
\label{pro-im}
The following statements are equivalent.

\begin{itemize}
\item[(i)] $\|M_J\|\approx 1$ for $J\in \mathcal L(Y_{\bf M})$.
\item[(ii)] $\|M_J\|\approx 1$ for $J\in \U$.
\item[(iii)] There exists $i\in \Lambda$ such that
$\|M_J^{(i)}\|\approx 1$ for $J\in \U$.
\item[(iv)] There exists $i\in \Lambda$ such that
$v_i^\T M_J^{(i)}u_i=1$ for $J\in \U$.
\item[(v)] There exists $i\in \Lambda$ such that
$v_i^\T M_J^{(i)}u_i=1$ for $J\in {\mathcal J}$.

\end{itemize}
 \end{pro}
\begin{proof}
We divide the proof into small steps.

{\sl Step 1. (i) $\Leftrightarrow$ (ii).}  Since $\U\subset \mathcal L(Y_{\bf M})$, the direction (i) $\Rightarrow$ (ii) is trivial. The reverse direction  follows immediately from Lemma \ref{lem-5.0}.

{\sl Step 2. (ii) $\Rightarrow$ (iii).} Suppose (ii) holds, that is, there exists a constant $C>0$ such that $$C^{-1}\leq \|M_J\|\leq C$$ for all $J\in \U$. Clearly we have $\|M_J^{(i)}\|\leq \|M_J\|\leq C$ for all $J\in \U$ and  $i\in \Lambda$.

Next we claim that there exists $i\in \Lambda$ such that $\|M_J^{(i)}\|\geq C^{-1}$ for all $J\in \U$. Clearly the claim implies (iii).  Suppose on the contrary that the claim is not true. Then for any $i\in \Lambda$, we can choose some $I_i\in \U$ such that
$$
\|M_{I_i}^{(i)}\|< C^{-1}.
$$
Construct a finite subset $\U_1$ of $\U$ by  $$\U_1=\mathcal J\cup \{I_i:\; i\in \Lambda\}, $$
and consider the new tuple  ${\bf N}:=(M_W)_{W\in \U_1}$ of non-negative matrices.  Let $\mu$ be the Parry measure on the full shift space $(\U_1)^\N$ over the alphabet $\U_1$.  Since the concatenation of any elements  of ${\mathcal U}_1$ is  in ${\mathcal U}$,   by (ii), we have $C^{-1}\leq \|M_{W_1\cdots W_n}\|\leq C$ for any $W_1,\ldots, W_n\in \U_1$. It follows that $\lambda({\bf N}, \mu)=0$, where  $\lambda({\bf N}, \mu)$ stands for the Lyapunov exponent of ${\bf N}$ with respect to $\mu$.
By the construction of $B$ and Lemma \ref{lem-3.7}, ${\bf N}^{(i)}:=\left(M_W^{(i)}\right)_{W\in \U_1}$ is positively irreducible whenever $i \in \Lambda$; otherwise, it consists only of the zero matrix ${\bf 0}$.

By Proposition \ref{pro-FK}, there exists $i\in \Lambda$ such that
\begin{equation}
\label{e-lambda}
\lambda\left({\bf N}^{(i)}, \mu\right)=0.
\end{equation}
 Since $\left\|M_{W_1}^{(i)}\cdots M^{(i)}_{W_n}\right\|\leq\|M_{W_1\cdots W_n}\|\leq  C$ for any $W_1,\ldots, W_n\in \U_1$, applying Lemma \eqref{lem-4.6} to the tuple ${\bf N}^{(i)}$ yields
$$
\left\|M_{W_1}^{(i)}\cdots M^{(i)}_{W_n}\right\|\geq C^{-1} \quad \mbox{ for any }W_1,\ldots, W_n\in \U_1,
$$
which contradicts $\left\|M_{I_i}^{(i)}\right\|< C^{-1}$. This proves the claim, and hence (iii) holds.

{\sl Step 3. (iii) $\Rightarrow$ (i).}  Suppose (iii) holds for some $i\in \Lambda$. Then
$$
\|M_J\|\geq \left\|M_J^{(i)}\right\|\succcurlyeq 1 \quad \mbox{ for }J\in \U.
$$
Applying Lemma \ref{lem-5.0}, we obtain $\|M_J\|\succcurlyeq 1$ for $J\in \mathcal L(Y_{\bf M})$.  Then (i) follows by Proposition \ref{lower bound implies upper bound}.

{\sl Step 4.  (iii) $\Longleftrightarrow$  (iv)}. Since $u_i, v_i$ are strictly positive vectors, we see that (iv) implies (iii).  Below we show that (iii) implies (iv).

  Suppose that (iii) holds for some $i\in \Lambda$. Let ${\mathcal S}=\{M_J^{(i)}:\; J\in \U\}$. Clearly ${\mathcal S}$ is a multiplicative semigroup, so are  ${\rm co}({\mathcal S})$ and $\overline{{\rm co}}({\mathcal S})$. By definition, we see that $(B^{(i)})^n\in {\rm co}({\mathcal S})$ for $n\in \N$. Therefore
 by Lemma \ref{lem-5.2}, $\|(B^{(i)})^n\|\approx 1$ for $n\in \N$.
  Thus $\rho(B^{(i)})=
\lim_{n\to \infty} \|(B^{(i)})^n\|^{1/n}=1$.  Since $(B^{(i)})^n$ is positively irreducible, by the Perron-Frobenius theory (see e.g. \cite[Theorem 8.6.1]{HornJohnson1985}), we have
$$
\lim_{N\to \infty} \frac{1}{N}\sum_{n=1}^N \left(B^{(i)}\right)^n=\lim_{N\to \infty} \frac{1}{N}\sum_{n=1}^N \left(\rho\left(B^{(i)}\right)^{-1}B^{(i)}\right)^n=u_iv_i^\T.
$$
It follows that $u_iv_i^\T\in \overline{{\rm co}}({\mathcal S})$. Since $\overline{{\rm co}}({\mathcal S})$ is a multiplicative semigroup, by Lemma \ref{lem-5.2}, we have
\begin{equation}
\label{e-UV}
\left \| \left( u_iv_i^\T M_J^{(i)} \right)^n\right\| \approx 1 \quad \mbox{ for }J\in \S, \; n\in \N.
\end{equation}
Since $\left(u_iv_i^\T M_J^{(i)}\right)^n= \left(v_i^\T M_J^{(i)} u_i\right)^{n-1} u_iv_i^TM_{J}^{(i)}$, \eqref{e-UV} implies that $v_i^\T M_J^{(i)} u_i=1$. Thus (iv) holds.

{\sl Step 5.  (iv) $\Leftrightarrow$ (v).} Clearly (iv) implies (v). Below we prove the reverse direction.

Suppose that (v) holds, that is, there exists $i\in \Lambda $ such that
\begin{equation}
\label{ev-5.3}
v_i^\T M_J^{(i)}u_i=1 \quad  \mbox{ for all } J\in {\mathcal J}.
\end{equation}
We need to show that $v_i^\T M_J^{(i)}u_i=1 \mbox{ for all } J\in \U$. To achieve this purpose, for $n\geq 1$ and $s\in \{1,\ldots, d\}$,  let $W_{n,s}$ be the smallest affine subset of $\R^{d_i}$ containing
the following set
$$\left\{ M_J^{(i)}u_i:\; J\in \A^*\backslash\{\varepsilon\},\; |J|\leq n,\; (M_J)_{s,1}>0 \right\}.$$
By \eqref{ev-5.3}, $W_{d^2, 1}$ is contained in the hyperplane  $\{u\in \R^{d_i}:\; v_i^\T u=1\}$.  Hence to show that $v_i^\T M_J^{(i)}u_i=1 \mbox{ for all } J\in \U$, it suffices to show that
\begin{equation}
\label{e-desire}
W_{n,1}=W_{d^2,1} \quad \mbox{ for all } n>d^2.
\end{equation}

By definition, we see that $W_{n+1,s}\supset W_{n,s}$ for all $n,s$, and moreover
\begin{equation}
\label{ev-5.5}\dim W_{n+1,s}> \dim W_{n,s}\quad \mbox{ if }\; W_{n+1,s}\neq W_{n,s}.
\end{equation}
 Let
$r_n=\sum_{s=1}^d\dim W_{n,s}$ for $n\geq 1$. Clearly the sequence $(r_n)$ is increasing and  bounded by $d^2$ from above.  Therefore, there exists $n_0\leq d^2$ such that
$r_{n_0+1}=r_{n_0}$. By  \eqref{ev-5.5}, we have   $W_{n_0+1,s}=W_{n_0,s}$ for all $1\leq s\leq d$. Below we show that $W_{n,s}=W_{n_0,s}$ for all $n\geq n_0+1$ and $1\leq s\leq d$, which implies \eqref{e-desire}.

 For this purpose, it is enough to show that if  for some $n\geq 1$, $W_{n+1,s}=W_{n,s}$ for all $1\leq s\leq d$, then  $W_{n+2,s}=W_{n+1,s}$ for all $1\leq s\leq d$. To prove this, suppose
 $W_{n+1,s}=W_{n,s}$ for all $1\leq s\leq d$. Fix $s\in \{1,\ldots, d\}$ and  $J=j_1\cdots j_{n+2}\in \A^{n+2}$ so that $(M_J)_{s,1}>0$.   Then there exists $p\in \{1,\ldots,d\}$ such that $(M_{j_1})_{s,p}>0$ and  $(M_{j_2\cdots j_{n+2}})_{p,1}>0$. Hence  $M_{j_2\cdots j_{n+2}}^{(i)}u_i\in W_{n+1,p}=W_{n,p}$. By \eqref{e-aff}, we can find $q\in \N$, $a_1,\ldots, a_q\in \R$ with $a_1+\cdots+a_q=1$,  and $J_1,\ldots, J_q \in \bigcup_{i=1}^n\A^i$ with $(M_{J_m})_{p,1}>0$ for $1\leq m\leq q$, such that
$$
M_{j_2\cdots j_{n+2}}^{(i)}u_i=\sum_{m=1}^q a_mM_{J_m}^{(i)}u_i.
$$
It follows that
$$
M_{j_1j_2\cdots j_{n+2}}^{(i)}u_i=\sum_{m=1}^q a_mM_{j_1}^{(i)}M_{J_m}^{(i)}u_i=\sum_{m=1}^q a_m M_{j_1J_m}^{(i)}u_i.
$$
Noticing that $(M_{j_1J_m})_{s,1}\geq (M_{j_1})_{s,p}(M_{J_m})_{p,1}>0$, the above relation yields that  $M_{J}^{(i)}u_i\in W_{n+1,s}$. Letting $J$ run over all elements in $\A^{n+2}$ with $(M_J)_{s,1}>0$,  we get
$W_{n+2,s}\subset W_{n+1,s}$, and so $W_{n+2,s}=W_{n+1,s}$. This completes the proof of the proposition.
\end{proof}

\begin{rem}Here we give an alternative proof of the direction (ii) $\Rightarrow$ (iii) by applying the results of Protasov and Voynov in \cite{ProtasovVoynov2014}.  Suppose (ii) holds. Then the semigroup $\{M_J:\; J\in {\mathcal U}\}$ has  constant spectral radius. By  \cite[Theorem 1]{ProtasovVoynov2014}, there exists $i$ such that  the semigroup $\{M_J^{(i)}:\; J\in {\mathcal U}\}$ is positively irreducible and has constant spectral radius. As it is pointed out in \cite{ProtasovVoynov2014}, for positively irreducible semigroups, the constant spectral radius is equivalent to boundedness from above and from below, from which (iii) follows.
\end{rem}

Now we are ready to prove Theorem \ref{thm-1.4}.
\begin{proof}[Proof of Theorem \ref{thm-1.4}]
 It follows directly from Proposition \ref{pro-im}.
\end{proof}

\section{The proof of Theorem \ref{thm-1.6}}
\label{S-new}
In this section we prove Theorem \ref{thm-1.6}. Let $P({\bf M}, \cdot)$ be the pressure function associated with ${\bf M}$ (see \eqref{e-pressure}). We first give a lemma.
\begin{lem}
\label{lem-derivative}
Suppose that ${\bf M}$ is irreducible or positively irreducible. Then the function $q\mapsto P({\bf M}, q)$ is differentiable over $(0,\infty)$ with derivative
\begin{equation}
\label{e-derivative}
P'({\bf M}, q)=\lambda({\bf M}, \nu_q),
\end{equation}
where $\nu_q$ is the equilibrium state for $({\bf M}, q)$ and $\lambda({\bf M}, \nu_q)$ is the Lyapunov exponent of ${\bf M}$ with respect to $\nu_q$ (see Section~\ref{S-pressure}).
\end{lem}
\begin{proof}
Let $q>0$, and let $\I_q$ be the collection of all equilibrium states for $({\bf M}, q)$. By Theorem \ref{thm-matrix},  $\I_q=\{\nu_q\}$ is a singleton.  Now \eqref{e-derivative}  follows from  the Ruelle-type derivative
formula of pressure functions obtained in \cite[Theorem 1.2]{Feng2004}:
$$
P'({\bf M}, q-)=\inf\{\lambda({\bf M},\mu): \; \mu\in \I_q\},\quad
P'({\bf M}, q+)=\sup\{\lambda({\bf M}, \mu): \; \mu\in \I_q\}.
$$
We remark that  although \cite[Theorem 1.2]{Feng2004} only deals with non-negative matrices, the proof given there works for arbitrary matrices.
\end{proof}

\begin{proof}[Proof of Theorem \ref{thm-1.6}] Let $\A=\{1,\ldots, k\}$. For $n\in \N$, set  
$$
\Omega_n=\{I\in \A^n:\; M_I\neq {\bf 0}\} \quad \mbox{and}\quad t_n=\#\Omega_n. 
$$
Clearly we have $t_{n+m}\leq t_nt_m$ and thus the following limit exists:
$$
\lim_{n\to \infty}\frac{1}{n}\log t_n=: h.
$$

Next we prove the following three properties are equivalent: 
\begin{itemize}
\item[(i)]
${\bf M}$ has a uniform Lyapunov exponent modula $0$; 
\item[(ii)] $P({\bf M}, \cdot)$ is affine on $(0,\infty)$;
\item[(iii)]$P({\bf M}, \cdot)$ is affine on $(a,b)$ for some $0<a<b<\infty$;
\end{itemize}
Since (ii)$\Rightarrow$(iii) is trivial, it suffices to prove the directions (i)$\Rightarrow$(ii) and (iii)$\Rightarrow$(i). 

We first prove (i)$\Rightarrow$(ii). Suppose that ${\bf M}$ has a uniform Lyapunov exponent modula $0$. Then there exists a constant $u\in \R$ such that 
$$
\|M_I\|\approx e^{un}\quad\mbox{ for } n\in \N, \; I\in \Omega_n.
$$
Hence for given $q>0$, 
$$
\sum_{I\in \A^n}\|M_I\|^q=\sum_{I\in \Omega^n}\|M_I\|^q \approx t_n e^{uqn},
$$
which implies $P({\bf M}, q)=h+uq$. Hence $P({\bf M}, \cdot)$ is affine on $(0,\infty)$. 

Next we prove (iii)$\Rightarrow$(i).   Suppose that $P({\bf M}, \cdot)$ is affine on some finite interval $(a,b)\subset (0,\infty)$. Then there exist $h_1, u_1\in \R$ such that
 $$P({\bf M}, q)=h_1+u_1q$$ for $q\in (a,b)$. By Lemma \ref{lem-derivative}, we have 
 $$
 u_1=P'({\bf M}, q)=\lambda({\bf M}, \nu_q)\qquad \mbox { for }q\in (a,b),
 $$
 where  $\nu_q$ is the equilibrium state for $({\bf M}, q)$ (thus $P({\bf M}, q)=h_{\nu_q}(\sigma)+q\lambda({\bf M}, \nu_q)$). Hence we have  
 $$
 \lambda({\bf M}, \nu_q)=u_1,\quad h_{\nu_q}(\sigma)=h_1 \qquad \mbox { for all }q\in (a,b).
 $$ 
Therefore for any $q_1, q_2\in (a, b)$, $\nu_{q_1}$ is an equilibrium state for $({\bf M}, q_2)$ since $$P({\bf M}, q_2)=h_1+u_1q_2=h_{\nu_{q_1}}(\sigma)+q_2\lambda({\bf M}, \nu_{q_1}).$$ However,  $({\bf M}, q_2)$  has a unique equilibrium state $\nu_{q_2}$, so we must have $\nu_{q_1}=\nu_{q_2}$.  Now fix two different elements $q_1, q_2$ in $(a,b)$. Since $\nu_{q_1}=\nu_{q_2}$, by Theorem \ref{thm-matrix}, we have
 $$
 \exp(-(h_1+u_1q_1)n)\|M_I\|^{q_1}\approx \exp(-(h_1+u_1q_2)n)\|M_I\|^{q_2} \qquad\mbox{ for }n\in \N,\; I\in \Omega_n,
 $$
which implies  $\|M_I\|\approx \exp(u_1n)$ for $n\in \N$ and $I\in \Omega_n$, that is, ${\bf M}$ has a uniform Lyapunov exponent modulo $0$.  This completes the proof of  (iii)$\Rightarrow$(i).

Now suppose that ${\bf M}$ has a uniform Lyapunov exponent modulo $0$. Then $P({\bf M}, \cdot)$ is affine on $(0,\infty)$ and thus \eqref{e-PC} holds.  

Conversely, suppose \eqref{e-PC} holds. By  convexity, $P({\bf M}, \cdot)$ is affine on $[2,6]$, which implies that  ${\bf M}$ has a uniform Lyapunov exponent modulo $0$. This completes the proof of the theorem.
\end{proof}

\section{Absolute Continuity of self-similar measures with finite type condition}
\label{S-5}
This section is devoted to the study of an extended version of Question \ref{ques-2}.

Let $\{S_j \}_{j=1}^m$ be a family of contractive similitudes on $\R$ of the form \eqref{e-IFS}.
Let $K$ denote the self-similar set generated by $\{S_j\}_{j=1}^m$ (cf. \cite{Hutchinson1981}), that is, $K$ is the unique non-empty compact set in $\R$ such that
$$
K=\bigcup_{j=1}^m S_j(K).
$$

Given a probability weight $\{p_j\}_{j=1}^m$, let  $\mu$ be the self-similar measure generated by $\{S_j\}_{j=1}^m$ and $\{p_j\}_{j=1}^m$.   It is supported on $K$, and contains no atoms (see e.g. \cite[Proposition 2.2]{FengLau2009}). As a  well-known fact,  $\mu$ is either singular or absolutely continuous with respect to  ${\mathfrak L}^1$, the Lebesgue measure
on $\R$ (see e.g. \cite[Proposition 3.1]{PeresSchlagSolomyak1998} for a proof). A similar argument  yields that  $\mu$ is also either singular or absolutely continuous with respect to $\mathcal{H}^s \big|_K$, where $$s=\dim_H K$$ is the Hausdorff dimension of $K$, $\mathcal{H}^s$  stands for the $s$-dimensional Hausdorff measure, and $\mathcal{H}^s \big|_K$ denotes the restriction of $\mathcal{H}^s$ on $K$. The reader is referred to \cite{Falconer2003, Mattila1995} for the definitions of Hausdorff dimension and Hausdorff measures.  Below we will provide criteria to determine these dichotomies under an additional separation assumption on $\{S_j\}_{j=1}^m$.

Write $S_J=S_{j_1}\circ \cdots\circ S_{j_n}$ for $J=j_1\cdots j_n$.
\begin{de}
\label{de-6.1}
We say that  $\{S_j\}_{j=1}^m$ satisfies the finite type condition if there is a finite set $\Gamma$ of non-negative numbers such that for each integer $n >0$ and any two words of indices $J = j_1 \cdots j_n$ and $J' = j_1'\cdots j_n'$,
\[
\text{ either } \ \ \rho^{-n} |S_J(0) - S_{J'}(0)| > c \ \ \text{ or } \ \ \rho^{-n} | S_J(0) - S_{J'}(0) | \in \Gamma,
\]
where $c: = (1-\rho)^{-1} (b_m- b_1)$.
\end{de}

The above definition of finite type condition was adopted from \cite{Feng2003}, and is slightly stronger than the one introduced by Ngai and Wang \cite{NgaiWang2001}. \footnote{In \cite{LauNgaiRao2001}, Lau, Ngai and Rao introduced an essentially identical separation condition called weak separation condition*.}
The finite type condition includes many interesting overlapping cases. For instance, if $\rho$ is the reciprocal of a Pisot number $\beta$ and $b_j\in {\Bbb Q}[\beta]$ for $j=1,\ldots, m$, where $\Q[\beta]$ stands for the field of $\beta$ over ${\Bbb Q}$, then $\{\rho x + b_j\}_{j=1}^m$ satisfies the finite type condition (see e.g.~\cite{NgaiWang2001}). Recall that $\beta>1$ is called  a {\it Pisot number} if $\beta$ is an algebraic integer  so that all its algebraic conjugates are less than $1$ in modulus.

 It is known (cf.~\cite{Nguyen2002}) that the finite type condition  implies the weak separation condition introduced by Lau and Ngai in \cite{LauNgai1999}. Hence due to   \cite[p. 3535]{Zerner1996}, if $\{S_j\}_{j=1}^m$ satisfies the finite type condition, then
\begin{equation}
\label{e-6.1}
0 < \mathcal{H}^s(K) < \infty;
\end{equation}
 moreover
\begin{equation}
\label{e-5.8}
\mathcal{H}^s(K \cap [x-r, x+r]) \approx r^s, \quad \mbox{ for }  x \in K,\; 0<r<1.
\end{equation}

It is known that under the assumption of finite type condition, the distribution of $\mu$ can be characterized through symbolic dynamics and  matrix products (cf.~\cite{Feng2003, Lalley1998}). Below we describe the characterization given in \cite{Feng2003}.

In \cite{Feng2003}, Feng constructed an irreducible subshift of finite type $\Sigma_A$ over a finite alphabet $\{1,\ldots, k\}$, a positively irreducible tuple ${\bf M}=(M_1,\ldots, M_k)$ of non-negative $d\times d$ matrices
 for certain $d$,   and a family of closed intervals $\{\Delta_I\}_{I\in {\mathcal L}(\Sigma_A)}$, where $\mathcal L(\Sigma_A)$ denotes the collection of all finite admissible words associated with $\Sigma_A$
 including the empty word $\varepsilon$ (see Section \ref{S-3.1}),  such that the following properties (C1)-(C5) hold:

\begin{itemize}
\item[(C1)]  $\{\Delta_I\}_{I\in {\mathcal L}(\Sigma_A)}$ has a nested structure, in the sense that,  for each $n\in \N$,  ${\rm int}(\Delta_{I})$ ($I\in \mathcal L_n(\Sigma_A)$) are disjoint subintervals of $\Delta_\varepsilon$, where ${\rm int}(A)$ stands for the interior of $A$; and moreover $\Delta_{i_1\cdots i_n}\subseteq \Delta_{i_1\cdots i_{n-1}}$ for any $i_1\cdots i_n\in  \mathcal L_n(\Sigma_A)$, where $\mathcal L_n(\Sigma_A)$ denotes the collection of admissible words of length $n$.
\item[(C2)] The lengths of  $\Delta_I$'s satisfy
$$
|\Delta_I|\approx \rho^n \quad \mbox{ for } n\in \N,\; I\in \mathcal L_n(\Sigma_A).
$$

\item[(C3)] $K\cap \Delta_\epsilon=K\cap \left(\bigcup_{I\in \mathcal  L_n(\Sigma_A)}\Delta_I\right)$ for any $n\in \N$. Moreover the endpoints of $\Delta_I$ are contained in $K$ for any $I\in \mathcal L(\Sigma_A)$.

\item[(C4)] $\mu(\Delta_I)\approx \|M_{i_1}\cdots M_{i_n}\| \quad \mbox{ for } n\in \N,\; I=i_1\cdots i_n\in \mathcal L_n(\Sigma_A)$.

\item[(C5)] For $i_1\cdots i_n\in \{1,\ldots,k\}^n$, $M_{i_1}\cdots M_{i_n}\neq {\bf 0}$ if and only if  $i_1\cdots i_n\in \mathcal L_n(\Sigma_A)$.

\end{itemize}

 It can be proved that the properties (C2)-(C3) imply that
 \begin{equation}
 \label{e-dim}
 s:=\dim_HK= \lim_{n\to \infty} \frac{\log \#   (\mathcal L_n(\Sigma_A))}{\log \rho^{-n}}=\frac{h_{\rm top}(\Sigma_A)}{\log (1/\rho)}.
 \end{equation}

Now we are ready to state the main result of this section.

\begin{thm}
\label{thm-6.1} Assume that $\{S_j\}_{j=1}^m$ satisfies the finite type condition. Let ${\bf M}=(M_1,\ldots, M_k)$ be constructed as above. Let $s=\dim_HK$. Then the following statements hold: \begin{itemize}
 \item[(i)] $\mu\ll {\mathcal H}^s|_K$ if and only if ${\bf M}$ has a uniform Lyapunov exponent modulo $0$.
 \item[(ii)]  $\mu\ll {\mathfrak L}^1$
if and only if $h_{\rm top}(\Sigma_A)=\log (1/\rho)$  and ${\bf M}$ has a uniform Lyapunov exponent modulo $0$.
 \end{itemize}
\end{thm}

\begin{rem}
{\rm
\begin{itemize}
 \item[(i)] In \cite[Theorem 1.3]{LauNgaiRao2001}, Lau, Ngai and Rao proved that, under a more general  assumption on $\{S_j\}_{j=1}^m$, $\mu$ is absolutely continuous with respect to ${\mathfrak L}^1$ if and only if certain constructed matrix has spectral radius $\rho$. Theorem \ref{thm-6.1}(ii) provided an alternative approach in deciding the type of $\mu$, which is checkable by Theorem \ref{thm-1.4}.
\item[(ii)] In \cite[Proposition 3.19]{HareHareNg2016}, Hare, Hare and Ng gave a sufficient condition  (in terms of
certain growth rate of matrix products) for $\mu$ to be absolutely continuous with respect to ${\mathcal H}^s|_K$, without indicating how to check  that condition.
\end{itemize}
}
\end{rem}

Let $P({\bf M})$ be the topological pressure of ${\bf M}$ (cf. \eqref{e-pM}), and $\nu$ the equilibrium state for  $({\bf M},1)$ (see Section~\ref{S-pressure}). Before proving Theorem \ref{thm-6.1}, we first give the following.

\begin{lem}
\label{lem-pressure}
 The following properties hold:
\begin{itemize}
\item[(i)] $P({\bf M})=0$.
\item[(ii)] $\nu$ satisfies
 \begin{equation}
 \label{e-nu}
 \nu([I])\approx \|M_{i_1}\cdots M_{i_n}\|\quad  \mbox{ for }n\in \N,\; I=i_1\cdots i_n\in  \mathcal L_n(\Sigma_A).
 \end{equation}
\item[(iii)] $\nu$ has no atoms.
\item[(iv)] ${\bf M}$ has a uniform Lyapunov exponent modulo $0$ if and only if $\nu$ is the Parry measure on $\Sigma_A$.
\end{itemize}
\end{lem}

\begin{proof}
 To prove (i),  recall that $\mu$ is supported on $K$ and has no atoms. By (C5), (C4), (C1) and (C3), we have
 \begin{eqnarray*}
 \sum_{i_1\cdots i_n\in \{1,\ldots,k\}^n}\|M_{i_1}\cdots M_{i_n}\|& = & \sum_{i_1\cdots i_n\in \mathcal L_n(\Sigma_A)}\|M_{i_1}\cdots M_{i_n}\|\\
 &\approx & \sum_{i_1\cdots i_n\in \mathcal L_n(\Sigma_A)}\mu(\Delta_{i_1\cdots i_n})\\
 &=& \mu(\Delta_\varepsilon),
 \end{eqnarray*}
 which implies that $P({\bf M})=0$. This proves (i).  Property (ii) just follows from (i) and  Theorem \ref{thm-matrix}.  To see (iii), recall that $\mu$ has no atoms. This implies
 $\mu(\Delta_{i_1\cdots i_n})\to 0$ as $n\to \infty$. By (C4) and \eqref{e-nu}, we have $\nu([i_1\cdots i_n])\to 0$ as $n\to \infty$, from which (iii) follows.

  Next we prove (iv).   In one direction, suppose that $\nu$ is the Parry measure on $\Sigma_A$.  By Theorem \ref{thm-Parry}, $\nu([I])\approx e^{|I|h_{\rm top}(\Sigma_A)}$ for  $I\in \mathcal L(\Sigma_A)$, which together with \eqref{e-nu} and (C5) yields that
  ${\bf M}$ has a uniform Lyapunov exponent modulo $0$. In the other direction, suppose that ${\bf M}$ has a uniform Lyapunov exponent modulo $0$. By (C5) and \eqref{e-nu}, there exists $\lambda\in \R$ so that
  $\nu([I])\approx e^{n\lambda}$ for $I=i_1\cdots i_n\in  \mathcal L_n(\Sigma_A)$.  This implies that $e^{n\lambda}\cdot\# (\mathcal L_n(\Sigma_A))\approx 1$, and so $\lambda=-h_{\rm top}(\Sigma_A)$ by \eqref{e-p1}. Hence $\nu([I])\approx e^{-|I| h_{\rm top}(\Sigma_A)}$ for  $I\in  \mathcal L(\Sigma_A)$. By Theorem \ref{thm-Parry}, $\nu$ is the Parry measure on $\Sigma_A$. This completes the proof.
  \end{proof}

  \begin{proof}[Proof of Theorem \ref{thm-6.1}]
Let $\nu$ be the equilibrium state for $({\bf M},1)$. By Lemma \ref{lem-pressure} (iv),  ${\bf M}$ has a uniform Lyapunov exponent modulo $0$ if and only if $\nu$ is the Parry measure on $\Sigma_A$. Hence to prove part (i) of the theorem, it is equivalent to show  that $\mu \ll {\mathcal H}^s|_K$ if and only if  $\nu$ is the Parry measure on $\Sigma_A$.

 First assume that   $\nu$ is the Parry measure on $\Sigma_A$.  By (C4), \eqref{e-nu}  and \eqref{e-dim},
 $$
 \mu(\Delta_{i_1\cdots i_n})\approx \|M_{i_1}\cdots M_{i_n}\|\approx \nu([i_1\cdots i_n])\approx  e^{-nh_{\rm top}(\Sigma_A)} = \rho^{s n}
 $$
 for $i_1\cdots i_n\in \mathcal L_n(\Sigma_A)$.
 Thus by \eqref{e-5.8} we have
 \begin{equation}
 \label{e-cite1}
 \mu(\Delta_\varepsilon\cap[x-\rho^n,x+\rho^n])\approx \rho^{ns}\approx {\mathcal H}^s|_K ([x-\rho^n,x+\rho^n])
 \end{equation}
  for $n\in \N$ and $x\in K\cap \Delta_\epsilon$,
  which implies  that $\mu|_{\Delta_\varepsilon}\ll  {\mathcal H}^s|_K$. Since $\mu$ is either purely singular or absolutely continuous with respect to ${\mathcal H}^s|_K$, we have   $\mu\ll {\mathcal H}^s|_K$.

 Next assume that  $\mu\ll {\mathcal H}^s|_K$. Then $\dim_H\mu=s$, where $\dim_H\mu$ stands for the Hausdorff dimension of $\mu$ (cf. \cite{FanLauRao2002}). Define $\pi: \Sigma_A\to K\cap \Delta_\epsilon$ by
 $$
 \{\pi({\bf i})\}=\bigcap_{n=1}^\infty \Delta_{i_1\cdots i_n},\quad \mbox{ for }{\bf i}=(i_n)_{n=1}^\infty.
 $$
  Let $\widetilde{\mu}=\nu\circ \pi^{-1}$. Since $\nu$ has no atoms by Lemma \ref{lem-pressure}(iii), we have by \eqref{e-nu},
 $$
 \widetilde{\mu}(\Delta_{i_1\cdots i_n})= \nu([i_1\cdots i_n])\approx \|M_{i_1}\cdots M_{i_n}\|\approx \mu(\Delta_{i_1\ldots i_n})$$
  for $n\in \N$ and $i_1\cdots i_n\in \L_n(\Sigma_A)$,
 which implies that there exists a constant $C>0$ such that $C^{-1}\mu|_{\Delta_{\varepsilon}}\leq  \widetilde{\mu}\leq C\mu|_{\Delta_{\varepsilon}}$. Hence $\dim_H \widetilde{\mu}=\dim_H\mu=s$.  It follows that (cf. \cite[Theorem 1.2]{FanLauRao2002}) that
    $$
 \liminf_{n\to \infty} \frac{\log \widetilde\mu([x-\rho^n,x+\rho^n])}{n\log \rho}\geq s \quad \mbox{ for $\widetilde{\mu}$-a.e.~$x\in \R$},
  $$
  equivalently,
   \begin{equation}
   \label{e-density}
 \liminf_{n\to \infty} \frac{\log \widetilde\mu([\pi {\bf i}-\rho^n, \pi {\bf i}+\rho^n])}{n\log \rho}\geq s \quad \mbox{ for $\nu$-a.e.~${\bf i} \in \Sigma_A$}.
  \end{equation}
By (C2), there exists $k_0\in \N$ such that for any ${\bf i}=(i_n)_{n=1}^\infty\in \Sigma_A$,
$$
\Delta_{i_1\cdots i_n}\subset [\pi {\bf i}-\rho^{n-k_0}, \pi {\bf i}+\rho^{n-k_0}],\quad n\in \N.
$$
 This together with  \eqref{e-density} yields that  for $\nu$-a.e.~${\bf i}=(i_n)_{n=1}^\infty\in \Sigma_A$,
  $$
  \liminf_{n\to \infty}
  \frac{\log \nu([i_1\cdots i_n])} {n\log \rho}
 \geq \liminf_{n\to \infty}  \frac{\log \widetilde{\mu}([\pi{\bf i}-\rho^{n-k_0}, \pi {\bf i}+\rho^{n-k_0}]}{n \log \rho} \geq s,
  $$
from which we obtain
  $$
  \liminf_{n\to \infty}
  \frac{-\log \nu([i_1\cdots i_n])} {n}
  \geq s\log (1/\rho)=h_{\rm top}(\Sigma_A)
  $$
for $\nu$-a.e.~${\bf i}=(i_n)_{n=1}^\infty\in \Sigma_A$. By the Shannon-McMillan-Breiman theorem (cf. \cite[p.~93]{Walters1982}), we have $h_\nu(\sigma)\geq h_{\rm top}(\Sigma_A)$,  which implies that $\nu$ is the Parry measure on $\Sigma_A$ by Theorem \ref{thm-Parry}. This proves (i).

Property (ii) just follows from (i), using the facts that $s=h_{\rm top}(\Sigma_A)/\log (1/\rho)=1$ and  $\mathcal H^1|_{\R}$ is equal to the Lebesgue measure ${\mathfrak L}^1$ on $\R$.
\end{proof}

We remark that the following corollary just follows from the  proof of Theorem \ref{thm-6.1}, together with an additional property that $\Delta_\epsilon\subset K$ whenever $\dim_HK=1$ (to be concise, we skip the proof of this property).
\begin{cor} \label{cor-6.1}

Under the condition of Theorem \ref{thm-6.1}, letting $s=\dim_HK$, then we have
\begin{itemize}
\item[(i)] $\mu\ll {\mathcal H}^s|_K\Longleftrightarrow \dim_H\mu=s\Longleftrightarrow \mbox{\eqref{e-cite1} holds for  $n\in \N$ and $x\in K\cap \Delta_\epsilon$}$.
\item[(ii)] $\mu\ll {\mathfrak L}^1\Longleftrightarrow \dim_H\mu=1 \Longleftrightarrow \frac{d\mu}{dx}\in (c_1,c_2) \mbox{ on }\Delta_\epsilon$ for some positive constants $c_1, c_2$.
\end{itemize}
\end{cor}

\begin{rem}
It is worth pointing out that  Ruiz \cite{Ruiz2008} proved the equivalence between $\mu\ll {\mathfrak L}^1$ and $\dim_H \mu=1$, in the special case  when $\{S_j\}_{j=1}^m$ is an integral iterated function system, i.e., $S_j$ is of the form $S_j(x)=\frac{1}{N}(x+d_j)$ with $N\in \N$ and $d_j\in \Z$.
\end{rem}
\section{Absolute continuity of a class of self-affine measures}
\label{S-6}

In this section we consider Question \ref{ques-2'}.
Let $A$ be a $d\times d$ integral expanding matrix and let $\{S_j\}_{j=1}^m$ be a family of affine maps on $\R^d$ given by
$$
S_j(x)=A^{-1}(x+d_j), \;\quad  j=1,\ldots, m,
$$
with $d_j\in {\Bbb Z}^d$. Let $K$ be the self-affine set generated by $\{S_j\}_{j=1}^m$ (cf.~\cite{Falconer2003}). Given a probability weight $\{p_j\}_{j=1}^m$, let $\mu$ be the self-affine measure generated by $\{S_j\}_{j=1}^m$ and $\{p_j\}_{j=1}^m$. That is, $\mu$ is the unique Borel probability measure on $\R^d$ such that
\begin{equation}
\label{e-affine}
\mu=\sum_{j=1}^m p_j\mu \circ S_j^{-1}.
\end{equation}
It is known that $\mu$ is supported on $K$. Similar to the self-similar case, $\mu$ is  either purely singular, or absolutely continuous with respect to the Lebesgue measure ${\mathfrak L}^d$ on $\R^d$.  Moreover if  $\mu\ll {\mathfrak L}^d$, then $\mu$ and ${\mathfrak L}^d|_K$ are equivalent (see \cite[Proposition~4.1(2)] {BarralFeng2013} and \cite[Proposition~22(3)]{Shmerkin2006}).

In this section we consider the problem of deciding whether $\mu$ is absolutely continuous. First let us recall a known criterion for this decision problem by using the approach of Fourier analysis.
For $\xi\in \R^d$, let
$$
\widehat{\mu}(\xi)=\int e^{-2\pi i \langle \xi,\; x \rangle\ } \; d\mu(x)
$$
be the Fourier transform of $\mu$, where $\langle\cdot,\cdot\rangle$ represents the standard inner product in $\R^d$.  By the self-affine property \eqref{e-affine}, one has $\widehat{\mu}(\xi)=\widehat{\mu}(\tilde{A}^{-1}\xi)P(\tilde{A}^{-1}\xi)$, where $\tilde{A}=A^{\T}$ and
\begin{equation}
\label{e-mask}
P(\xi):=\sum_{j=1}^m p_je^{-2\pi i\langle \xi,\; d_j \rangle}.
\end{equation}
It follows that
$$
\widehat{\mu}(\xi)=\prod_{n=1}^\infty P(\tilde{A}^{-n}\xi).
$$

The following result is known to the experts in the areas of self-affine tilings and wavelet theory.
\begin{pro}
\label{pro-self-affine}
 The following statements are equivalent:
\begin{itemize}
\item[(i)] $\mu$ is absolutely continuous with respect to ${\mathfrak L}^d$.
\item[(ii)] $\widehat{\mu}({\bf m})=0$ for any ${\bf m}\in \Z^d\setminus \{{\bf 0}\}$.
\item[(iii)] For any ${\bf m}\in \Z^d\setminus \{{\bf 0}\}$, there exists $n\in \N$ such that $P(\tilde{A}^{-n}{\bf m})=0$.
\item[(iv)] $\overline{\mu}$ is the Haar measure on $\R^d/\Z^d$, where $\overline{\mu}$ stands for  the push forward of $\mu$ under the canonical projection $\pi: \;\R^d\to \R^d/\Z^d$, i.e., $\overline{\mu}=\mu\circ \pi^{-1}$.
\end{itemize}
\end{pro}
 \begin{proof}
 It follows from the proof of \cite[Theorem 2.1]{LagariasWang1996} with minor modifications.
 \end{proof}

 \begin{rem}
 \label{rem-6.2}
 For the case $d=1$, Protasov \cite{Protasov2000} provided an efficient algorithm to decide whether (iii) of Proposition \ref{pro-self-affine} is fulfilled, and hence to decide whether $\mu$ is absolutely continuous.  This algorithm is essentially based on the fact that in the case $d=1$, the mask function $P$ defined in \eqref{e-mask} has at most finitely many (rational) zero points lying in $[0,1)$. In the higher dimensional case, since $P$ may have infinitely many (rational) zero points in $\R^d/\Z^d$, it is unlikely that Protasov's algorithm is still efficient.
\end{rem}

In this section, we will provide an algorithm to decide the absolute continuity of $\mu$ in the general high dimensional case.  Our starting point  is  the work of Deng, He and Lau \cite{DengHeLau2008} on the structure of $\mu$.

In \cite{DengHeLau2008}, the authors constructed a $\Z^d$-tile $T\subset \R^d$, which is the attractor of certain affine iterated function system $\{\psi_i(x)=A^{-n_0}(x+c_i)\}_{i=1}^\ell$, with $n_0\in \N$, $\ell=|\det(A)|^{n_0}$ and $c_i\in \Z^d$, such that
\begin{equation}
\label{e-mu-bdd}
\mu(\partial T+{e})=0\quad  \mbox{ for all }{e}\in \Z^d,
\end{equation}
where $\partial T$ stands for the boundary of $T$. 
Set $${\mathcal E}=\{{e}_1,\ldots, {e}_N\}=\{{e}\in \Z^d:\; K\cap (\mbox{int}(T)+{e})\neq \emptyset\}$$ and define the vector-valued measure $\pmb{\mu}$ on $T$  by
$$
\pmb{\mu}(E)=[\mu((E\cap T)+{e}_1), \ldots, \mu((E\cap T)+{e}_N)]^\T.
$$
For $J=j_1\cdots j_{n_0}\in \{1,\ldots, m\}^{n_0}$, set $p_J=p_{j_1}\cdots p_{j_{n_0}}$ and $d_J=\sum_{k=1}^{n_0} A^{n_0-k}d_{j_k}$. Define a tuple ${\bf M}=(M_1,\ldots, M_\ell)$ of $N\times N$ non-negative matrices by
$$
(M_k)_{i,j}=\left\{\begin{array}{cc} p_J & \mbox{ if }\; c_k+A^{n_0} e_i-e_j=d_J \;\mbox{ for some }J\in \{1,\ldots, m\}^{n_0},\\
0 & \mbox{ otherwise},
\end{array}
\right.
$$
where $1\leq k\leq \ell$, and $1\leq i,j\leq N$.   The following theorem is our starting point.

\begin{thm}  [{\cite[Theorems 1.1-1.2]{DengHeLau2008}}]
\label{thm-DHL}
 \begin{itemize}
\item[(i)]The tuple ${\bf M}$ is positively irreducible.
\item[(ii)] $\sum_{i=1}^\ell M_i$ is Markov, i.e., all its column sums are equal to $1$.
\item[(iii)] For any $I=i_1\ldots i_n\in \{1,\ldots, \ell\}^n$,
$$
\pmb{\mu}(\psi_{I}(T))=M_I\pmb{\mu}(T),
$$
where $\psi_I:=\psi_{i_1}\circ \cdots \circ \psi_{i_n}$ and $M_I:=M_{i_1}\cdots M_{i_n}$.
\item[(iv)]${\mathfrak{L}}^d(K)>0$ if and only if $M_I\neq {\bf 0}$ for every finite word $I$ on $\{1,\ldots, \ell\}$.
\end{itemize}
\end{thm}

\begin{rem}
\label{rem-L}
 Some equivalent conditions for $\mu$ to be absolutely continuous were given in {\cite[Proposition 3.8]{DengHeLau2008}}
in terms of
 joint spectral radius of matrix products. However, such conditions on  the joint spectral radius are undecidable in general (see \cite{BlondelTsitsiklis2000b}). One may see \cite{JiaLauZhou2001, LagariasWang1995, HareMorrisSidorovTheys2011} for some related works on the $L^1$-solutions of scaling equations and the joint spectral radius of matrix products.
\end{rem}

We say that the tuple ${\bf M}$ has a {\it uniform Lyapunov exponent} if there exists $\lambda\in \R$ such that $\|M_I\|\approx e^{\lambda n}$ for $n\in \N$ and $I\in  \{1,\ldots, \ell\}^n$.  One of the main results of this section is the following.
\begin{thm}
\label{thm-6.4}
The following statements are equivalent:
\begin{itemize}
\item[(i)] $\mu$ is absolutely continuous.
\item[(ii)]   ${\bf M}$ has a uniform Lyapunov exponent.
\end{itemize}
\end{thm}
\begin{proof}
Since $T$ is a self-affine $\Z^d$-tile of $\R^d$, there exists a Borel set $T'\subset T$ such that  $\mbox{int}(T')=\mbox{int}(T)$, and  $\R^d=\bigcup_{e\in \Z^d} (T'+e)$ with the union being disjoint; in other word, $T'$ is a fundamental domain of the torus $\R^d/\Z^d$. Let $\pi:\; \R^d\to \R^d/\Z^d$  be the canonical projection and $\overline{\mu}=\mu\circ\pi^{-1}$.

By \eqref{e-mu-bdd} and \eqref{e-affine}, one can derive that  $\mu(\psi_I(\partial T)+e)=0$ and hence $\mu(\psi_I(T)+e)=\mu(\psi_I(T')+e)$ for any $e\in \Z^d$ and any finite word $I$ on the alphabet $\{1,\ldots, \ell\}$. Combining it with Theorem \ref{thm-DHL}(ii) yields
\begin{equation}
\label{e-MI}
\overline{\mu}(\psi_I(T'))=\overline{\mu}(\psi_I(T))=(1,1,\ldots,1)M_I {\pmb \mu}(T)\approx \|M_I\|.
\end{equation}

Suppose that $\mu$ is absolutely continuous, by Proposition \ref{pro-self-affine}, $\overline{\mu}$ is the Haar measure on $\R^d/\Z^d$. Then $\overline{\mu}(\psi_I(T'))=\ell^{-|I|}$ and thus by \eqref{e-MI}, $\|M_I\|\approx \ell^{-|I|}$. Hence ${\bf M}$ has a uniform Lyapunov exponent.

Next suppose that $M$ has a uniform Lyapunov exponent. Then by \eqref{e-MI}, we have  $$\overline{\mu}(\psi_I(T'))\approx \overline{\mu}(\psi_J(T'))$$ for $n\in \N$ and $I, J\in \{1,\ldots, \ell\}^n$. It follows that $\overline{\mu}(\psi_I(T'))\approx \ell^{-|I|}={\mathfrak{L}}^d(\psi_I(T'))$, which implies that $\overline{\mu}$ is absolutely continuous with respect to the Haar measure on $\R^d/\Z^d$. Hence $\mu$ is absolutely continuous with respect to ${\mathfrak L}^d$.
\end{proof}

In the remaining part of this section, we prove the following additional property of $\mu$.
\begin{thm}
\label{thm-cts} $\mu$ is absolutely continuous if $\dim_H\mu=d$.
\end{thm}

First we give an equivalent condition for $\dim_H\mu=d$ in terms of measure-theoretic entropies. Recall that for a probability measure $\eta$ on $\R^d$ and a finite or countable  Borel partition ${\mathcal P}=\{C_1,\ldots, C_k,\ldots\}$ of $\R^d$, the  entropy of $\eta$ with respect to $\mathcal P$ is defined by
$$
H_\eta(\mathcal P)=-\sum_{k=1}^\infty \eta(C_k)\log \eta(C_k).
$$

\begin{lem}
\label{lem-important}
Let ${\mathcal Q}$ denote the partition $\{[0,1)^d+\alpha:\; \alpha\in \Z^d\}$. Set  ${\mathcal Q}_n:=A^{-n}Q$ for $n\in \N$. Then the limit
$$
h_\mu^*:=\lim_{n\to \infty}\frac{H_\mu({\mathcal Q}_n)}{n}
$$
exists. Furthermore   $\dim_H\mu=d$ if and only if  $h_\mu^*=\log |\det(A)|$.
\end{lem}
\begin{proof}
The first result follows directly from \cite[Theorem 2.3(i)]{FengHu2009}. The second one can be derived  from a formula of $\dim_H\mu$ established in \cite[Theorem 2.11(ii)]{FengHu2009}. To avoid introducing too many terminologies, we simply give the proof for the special case  when
$d=2$ and $A={\rm diag}(a,b)$ with $1<a<b$. The proof for the general case is similar in spirit.

In this special case, the formula of $\dim_H\mu$  given in \cite{FengHu2009} can be rewritten as
\begin{equation}
\label{e-LY}
\dim_H\mu=\left(\frac{1}{\log a}-\frac{1}{\log b}\right)H_1+\frac{h_\mu^*}{\log b},
\end{equation}
where
$$
H_1=\lim_{n\to \infty} \frac{H_\nu({\mathcal D}_n)}{n},
$$
here  $\nu$ is the push-forward of $\mu$ under the projection $\tau:\R^2\to \R$ given by $(x,y)\mapsto x$, that is, $\nu=\mu\circ \tau^{-1}$; and ${\mathcal D}_n=a^{-n}{\mathcal D}$ with
\begin{equation}
\label{e-cald}
{\mathcal D}=\{[0,1)+\beta:\; \beta\in \Z\}.
\end{equation}
 Again the existence of the limit in defining $H_1$ follows from \cite[Theorem 2.3(i)]{FengHu2009}.

 Next we claim that
 \begin{equation}
 \label{e-HV}
 h_\mu^*\leq \log (ab)\quad \mbox{  and } \quad H_1\leq \log a.
  \end{equation}
 We only prove the first inequality, the second one follows by a similar argument.
 Notice that $\mu$ is supported on $K$ which is compact. Set $$\widetilde{{\mathcal Q}}_n=\{Q\in {\mathcal Q}_n:\; \mu(Q)>0\}.$$
Since any member of $\widetilde{{\mathcal Q}}_n$ intersects $K$ and has volume $(ab)^{-n}$,  a simple volume argument yields that
  $$
  \# \widetilde{{\mathcal Q}}_n\leq C (ab)^n
  $$
 for some constant $C>0$ depending on the diameter of $K$.  Hence
 $$H_\mu({\mathcal Q}_n)=H_\mu(\widetilde{{\mathcal Q}}_n)\leq \log \left(\# \widetilde{{\mathcal Q}}_n\right)\leq n\log (ab)+\log C,$$
 from which the inequality $h_\mu^*\leq \log (ab)$ follows.

 Now by \eqref{e-LY} and \eqref{e-HV},  we have
 $$
 \dim_H\mu\leq \left(\frac{1}{\log a}-\frac{1}{\log b}\right)\log a +\frac{\log(ab)}{\log b}=2,
 $$
 and hence the condition  $\dim_H\mu=2$ holds if and only if that  $h_\mu^*=\log (ab)$ and $H_1=\log a$.

 To complete the proof, we need to show that $h_\mu^*=\log (ab)$ implies that $H_1=\log a$. To see this implication, suppose $h_\mu^*=\log (ab)$. For $n\in \N$, define two partitions ${\mathcal E}_{n}$ and ${\mathcal F}_n$ of $\R^2$ by
 $$
 {\mathcal E}_{n}=\{E\times \R:\; E\in a^{-n} \mathcal D\},\qquad {\mathcal F}_{n}=\{\R\times F:\; F\in b^{-n} \mathcal D\},
 $$
 where ${\cal D}$ is given as in \eqref{e-cald}.
 It is direct to check that $$\mathcal Q_n={\mathcal E}_{n}\vee {\mathcal F}_{n}:=\{C\cap D: C\in {\mathcal E}_{n}, D\in {\mathcal F}_{n}\}.$$
  Hence we have  \begin{equation}
 \label{entropy}H_\mu(\mathcal Q_n)\leq H_\mu({\mathcal E}_{n})+H_\mu(\mathcal F_n).
 \end{equation} (cf.~\cite[Theorem 4.3]{Walters1982}). Noticing that $H_\mu({\mathcal E}_{n})=H_\nu({\mathcal D}_{n})$, we have $$\lim_{n\to \infty}\frac{1}{n}H_\mu({\mathcal E}_{n})=H_1\leq \log a;$$ similarly, we can prove $\lim_{n\to \infty}(1/n)H_\mu({\mathcal F}_{n})\leq \log b$. Thus by \eqref{entropy}, we have
 $$
 \log(ab)=\lim_{n\to \infty}\frac{H_\mu(\mathcal Q_n)}{n}\leq  \lim_{n\to \infty}\frac{H_\mu({\mathcal E}_{n})}{n}+\lim_{n\to \infty}\frac{H_\mu({\mathcal F}_{n})}{n}\leq \log a+\log b,
 $$
 from which  the equality  $H_1= \log a$ follows.
  \end{proof}

\begin{rem}
We emphasize that in Lemma \ref{lem-important}, the assumption that  $A$ and $d_j$'s are integral is not needed.
\end{rem}

Let $\sigma$ be the left shift map on $\Sigma:=\{1,\ldots, \ell\}^\N$.

\begin{lem}
\label{lem-xi}
Let $h_\mu^*$ be defined as in Lemma \ref{lem-important}. Then
$$
h_\mu^*=\frac{h_\xi(\sigma)}{n_0},
$$
where $\xi$ is the equilibrium state for $({\bf M},1)$.
\end{lem}
\begin{proof}
Since $\sum_{i=1}^\ell M_i$ is Markov by Theorem \ref{thm-DHL}(ii), we have $\rho\left(\sum_{i=1}^\ell M_i\right)=1$ and hence $P({\bf M})=0$.  By Theorem \ref{thm-matrix}, $\xi$ is the unique ergodic invariant measure on $\Sigma$ so that
\begin{equation}
\label{e-xi}
\xi([I])\approx \|M_I\|,\quad I\in  \bigcup_{n=1}^\infty \{1,\ldots, \ell\}^n.
\end{equation}
Set $$c=\min\left\{\mbox{non-zero entries of } M_i:\; i=1,\ldots, \ell\right\}.$$
Clearly $c>0$ and
\begin{equation}
\label{e-MI-1}
\|M_I\|\geq c^{|I|} \quad \mbox{ for all $I$ with }M_I\neq {\bf 0}.
\end{equation}

Let $\pi:\; \R^d\to \R^d/\Z^d$  be the canonical projection and $\overline{\mu}=\mu\circ\pi^{-1}$. Let $T'\subset T$ be defined as in the proof of Theorem \ref{thm-6.4}. For $n\in \N$, set $\Sigma_n:=\{1,\ldots, \ell\}^n$. Then $\overline{\mu}$ is supported on $\bigcap_{n=1}^\infty\bigcup_{I\in \Sigma_n}\psi_I(T')$.
By \eqref{e-MI}, \eqref{e-xi} and \eqref{e-MI-1}, there exists a constant $t\geq 1$ such that
\begin{equation}
\label{e-et}
t^{-2}c^{|I|}\leq t^{-1}  \xi([I]) \leq \overline{\mu}(\psi_I(T'))\leq t \xi([I])\quad \mbox{ for  all }I\in  \bigcup_{n=1}^\infty \{1,\ldots, \ell\}^n.
\end{equation}

 Below, we  show that
\begin{equation}
\label{e-SMB}
\lim_{n\to \infty} \frac{1}{n}\sum_{I\in \Sigma_n}\left(-\overline{\mu}(\psi_I(T')) \log \overline{\mu}(\psi_I(T'))\right)=h_\xi(\sigma).
\end{equation}
To see this, by the Shannon-McMillan-Breiman theorem we obtain that, for any $\epsilon>0$, there exists $k(\epsilon)\in \N$ such that for all $n\geq k(\epsilon)$,
\begin{equation}
\label{e-estimate}
\sum_{I\in \Omega_{n,\epsilon}}\xi([I])<\epsilon,
\end{equation}
where
$$
\Omega_{n,\epsilon}:=\left\{I\in \Sigma_n:\; \left|\frac{\log \xi([I])}{(-n)}-h_\xi(\sigma)\right|>\epsilon\right\}.
$$
 By \eqref{e-et}
and the definition of $\Omega_{n,\epsilon}$, we have
$$
\left|\frac{\log \overline{\mu}(\psi_I(T'))}{(-n)}-h_\xi(\sigma) \right|\leq \left\{
\begin{array}{ll}
|\log c|+h_\xi(\sigma)+2 n^{-1}\log t &\mbox{ if }I\in \Sigma_{n},\\
\epsilon+n^{-1}\log t & \mbox{ if }I\in \Sigma_n\backslash \Omega_{n,\epsilon}.
\end{array}
\right.
$$
Hence
\begin{equation*}
\begin{split}
\sum_{I\in \Sigma_n}& \overline{\mu}(\psi_I(T'))
\left|\frac{\log \overline{\mu}(\psi_I(T'))}{(-n)}-h_\xi(\sigma) \right|\\
\leq & \left(\sum_{I\in \Omega_{n,\epsilon}} \overline{\mu}(\psi_I(T'))\right)\left(|\log c|+h_\xi(\sigma)+ \frac{2}{n}\log t\right)\\
\mbox{}&\qquad \qquad+ \left(\sum_{I\in \Sigma_n\backslash \Omega_{n,\epsilon}} \overline{\mu}(\psi_I(T'))\right)\left (\epsilon+\frac{1}{n}\log t\right)\\
\leq & t \left(\sum_{I\in \Omega_{n,\epsilon}} \xi([I])\right)  \left(|\log c|+h_\xi(\sigma)+\frac{2}{n}\log t\right)+  \left(\epsilon+\frac{1}{n}\log t\right)\\
\leq & t \epsilon  \left(|\log c|+h_\xi(\sigma)+\frac{2}{n}\log t\right)+  \left(\epsilon+\frac{1}{n}\log t\right),
\end{split}
\end{equation*}
which is bounded from above by $\tilde{c}\epsilon$ for certain positive constant $\tilde{c}$ when $n$ is large enough. Now \eqref{e-SMB} follows by letting $\epsilon\to 0$.

Next for $k\in \N$, construct $2$  partitions $\widetilde{\mathcal Q}_k, \widetilde{\mathcal P}_k$ of ${\rm supp}(\mu)$ by
\begin{align*}
\widetilde{\mathcal Q}_k &:=\{Q\in {\mathcal Q}_k:\; \mu(Q)>0\},\\
\widetilde{\mathcal P}_k &:=\left\{\bigcup_{i=1}^N \left(\psi_I(T')+e_i\right):\; I\in \Sigma_k\right\}.
\end{align*}
A simple geometric argument yields that there exists  $u\in \N$ (which is independent of $k$) such that each member in $\widetilde{\mathcal Q}_{kn_0}$ intersects at most $u$ members of  $\widetilde{\mathcal P}_k$, and vice versa. This implies that
$$
\left|H_\mu\left(\widetilde{\mathcal Q}_{kn_0}\right)-H_\mu\left(\widetilde{\mathcal P}_k\right)\right|\leq \log u.
$$
For a proof, see e.g. \cite[Lemma 4.6]{FengHu2009}.  Clearly $H_\mu(\widetilde{\mathcal Q}_k)=H_\mu({\mathcal Q}_k)$. Since  $$\mu\left(\bigcup_{i=1}^N \left(\psi_I(T')+e_i\right)\right)=\overline{\mu}(\psi_I(T'))$$ for each $I\in \Sigma_k$, we get
$$
H_\mu\left(\widetilde{\mathcal P}_k\right)=\sum_{I\in \Sigma_k} \left(-\overline{\mu}(\psi_I(T')) \log \overline{\mu}(\psi_I(T'))\right).
$$
Hence by \eqref{e-SMB}, we get $$\lim_{k\to \infty} \frac{1}{k} H_\mu(  {\mathcal Q}_{kn_0})= \lim_{k\to \infty} \frac{1}{k} H_\mu(  \widetilde{\mathcal Q}_{kn_0})=\lim_{k\to \infty} \frac{1}{k} H_\mu(  \widetilde{\mathcal P}_{k})=h_\xi(\sigma),
$$
from which we obtain $h_\mu^*=h_\xi(\sigma)/n_0$.
\end{proof}

\begin{proof}[Proof of Theorem \ref{thm-cts}]
Suppose that $\dim_H\mu=d$. By Lemma \ref{lem-important}, we have $h_\mu^*=\log |\det A|$. Thus by Lemma \ref{lem-xi}, we get $$h_\xi(\sigma)=n_0h_\mu^*=n_0\log |\det A|=\log \ell.$$
It follows that $\xi$ is the Parry measure on $\Sigma$, and hence  by \eqref{e-xi},
$$
\|M_I\|\approx \xi([I])=\ell^{-n}\mbox{ for }n\in \N \mbox{ and }I\in \{1,\ldots, \ell\}^n.
$$
Therefore, ${\bf M}$ has a uniform Lyapunov exponent. By Theorem \ref{thm-6.4}, $\mu$ is absolutely continuous.
\end{proof}

\section{Projections of Parry measures  under factor maps} \label{S-7}
This section is devoted to the study of Question \ref{ques-3}.

Let  $n, m\in \N$.
Let $\tau$ be a mapping from $\{1,\ldots, n\}$ to $\{1,\ldots, m\}$. Then $\tau$ induces a one-block mapping  $\pi: \{1,\ldots, n\}^\N\to \{1,\ldots, m\}^\N$ by
$$
\pi \left( (x_k)_{k=1}^\infty \right) = \left(\tau (x_k)_{k=1}^\infty \right),\\ \mbox{ for }
(x_k)_{k=1}^\infty\in \{1,\ldots, n\}^\N.
$$

 Let $(\Sigma_A,\sigma)$ be the subshift of finite type over $\{1,\ldots, n\}$,  associated with  a positively  irreducible  $0$-$1$ matrix  $A=(a_{i,j})_{1\leq i,j\leq n}$ (see Section \ref{S-3.1}). Then $Y=\pi(\Sigma_A)$ is an irreducible sofic shift.
Let $\mu$, $\nu$ denote the Parry measures on $\Sigma_A$ and $Y$, respectively (see Section \ref{S-3.2}). Question \ref{ques-3} asks  whether $\nu=\mu\circ \pi^{-1}$.

For each $\ell \in \{1,\ldots, m\}$, define an $n \times n$ matrix $E_\ell=((E_\ell)_{i,j})_{1\leq i,j\leq n}$ by
\[
(E_\ell)_{i,j} = \begin{cases}
a_{i,j} & \text{ if }\pi (j) = \ell, \\ 0 & \text{ otherwise.}
\end{cases}
\]
The main result of this section is the following.

\begin{thm}
\label{thm-5.1}
The tuple  ${\bf E}=(E_1,\ldots, E_m)$ is positively irreducible. Moreover, $\nu=\mu\circ \pi^{-1}$ if and only if ${\bf E}$ has a uniform Lyapunov exponent modulo $0$.
\end{thm}

To prove the above theorem, we first give a simple lemma.

\begin{lem}
\label{lem-5.1}
\begin{itemize}
\item[(i)] For  $y_1,\ldots, y_k\in \{1,\ldots, m\}$ and $i,j\in \{1,\ldots, n\}$, we have
\begin{align*}
(E_{y_1} \cdots E_{y_k})_{i,j} = \#  \{ x_1 \cdots x_k \in \L_k(\Sigma_{A}): \;\;  & \tau(x_\ell)=y_\ell \mbox{ for }1\leq \ell\leq k, \\ &  x_k =j   \text{ and } a_{i, x_1} = 1 \}.
\end{align*}

\item[(ii)] $E_{y_1} \cdots E_{y_k}\neq 0$ if and only if $y_1\cdots y_k\in {\mathcal L}(Y)$.

\item[(iii)] $\|E_{y_1} \cdots E_{y_k}\|\approx N(y_1\cdots y_k)$ for $y_1\cdots y_k\in {\mathcal L}(Y)$, where
\begin{equation}
\label{e-5.0}
N(y_1\cdots y_k):=\#  \{ x_1 \cdots x_k \in \L_k(\Sigma_{A}):\;  \tau(x_\ell)=y_\ell \mbox{ for }1\leq \ell\leq k\}.
\end{equation}

\item[(iv)] $\sum_{k=1}^mE_k=A$, and hence  ${\bf E}$ is positively irreducible.
\end{itemize}
\end{lem}

\begin{proof}
By the definition of $E_1, \ldots, E_m$, we have
\begin{eqnarray*}
(E_{y_1} \cdots E_{y_k})_{i,j}&=&\sum_{1\leq x_1,\ldots, x_{k-1}\leq n} (E_{y_1})_{i, x_1} (E_{y_2})_{x_1, x_2}\cdots (E_{y_k})_{x_{k-1}, j}\\
&=&\sum_{1\leq x_1,\ldots, x_{k-1}\leq n \atop{\tau(x_\ell)=y_\ell,\; 1\leq \ell\leq k-1 \atop{\tau(j)=y_k}}} a_{i, x_1} a_{x_1, x_2}\cdots a_{x_{k-1}, j},
\end{eqnarray*}
from which (i) follows.

Clearly (ii) follows from (i), and (iv) follows from the definitions of $E_k$'s. To see (iii),  one can directly  deduce from (i) that
$$N(y_1\cdots y_k)\leq  \|E_{y_1} \cdots E_{y_k}\|\leq n^2 N(y_1\cdots y_k).$$
\end{proof}

\begin{proof}[Proof of Theorem \ref{thm-5.1}] By Lemma \ref{lem-5.1}(iv), ${\bf E}$ is positively irreducible.
Let $$\alpha = \exp(h_{\rm top}(\Sigma_A)),\quad \beta = \exp(h_{\rm top}(Y)).$$ Since $\mu$ and $\nu$ are  the Parry measures on $\Sigma_A$ and $Y$, by Theorem \ref{thm-Parry}, we have
\begin{equation}
 \label{e-5.1}
 \mu([I]) \approx \alpha^{-k} \quad \mbox{ for } k\in \N \mbox{ and } I \in \L_k(\Sigma_A)
\end{equation}
and
\begin{equation}
\label{e-5.2}
 \nu([J]) \approx \beta^{-k}\quad \mbox{ for } k\in \N \mbox{ and } J \in \L_k(Y).
\end{equation}
Notice that for $J \in \L_k(Y)$, $$\mu\circ \pi^{-1}([J])=\sum_{I\in \L_k(\Sigma_A),\; \pi(I)=J}\mu([I]),$$
 where $\pi(I):=\tau(i_1)\cdots \tau(i_k)$ for $I=i_1\cdots i_k$.   By \eqref{e-5.1}, we have
\begin{equation}
\label{e-5.3}
\mu\circ \pi^{-1}([J]) \approx N(J) \alpha^{-|J|} \quad \mbox{ for }J\in \L(Y),
\end{equation}
where $N(J)$ is defined as in \eqref{e-5.0}.

Suppose $\nu=\mu\circ \pi^{-1}$. Then by \eqref{e-5.2} and \eqref{e-5.3}, we have
\begin{equation}
\label{e-5.4}
 N(J) \approx (\alpha / \beta)^{|J|} \quad \mbox{ for }J\in \L(Y).
\end{equation}
By Lemma \ref{lem-5.1}(iii), we obtain  $$\|E_{y_1}\cdots E_{y_k}\|\approx N(y_1\cdots y_k)\approx (\alpha/\beta)^{k}$$ for $y_1\cdots y_k\in \L(Y)$. This together with Lemma \ref{lem-5.1}(ii) shows that
${\bf E}$ has a uniform Lyapunov exponent modulo $0$.

Conversely,
suppose that  ${\bf E}$ has a uniform Lyapunov exponent modulo $0$. Then by Lemma \ref{lem-5.1}(ii), there exists $\lambda\in \R$ such that
$$\|E_{y_1}\cdots E_{y_k}\|\approx e^{k\lambda} \quad \mbox{ for } k\in \N \mbox{ and }y_1\cdots y_k\in \L(Y).$$
Hence by Lemma \ref{lem-5.1}(iii),  $N(J)\approx e^{\lambda |J|}$ for $J\in \L(Y)$. By \eqref{e-5.3}, we have
\begin{equation}
\label{e-5.5}
\mu\circ \pi^{-1}([J]) \approx (e^\lambda\alpha^{-1})^{|J|} \quad \mbox{ for }J\in \L(Y).
\end{equation}
This yields $$1=\sum_{J\in \L_k(Y)}\mu\circ \pi^{-1}([J]) \approx \# (\L_k(Y)) (e^\lambda\alpha^{-1})^{k}\quad  \mbox{ for }k\in \N.$$
It implies $e^\lambda=\alpha/\beta$ since $\lim_{k\to \infty}(1/k) \log \# (\L_k(Y))=h_{\rm top}(Y)=\log \beta$.  Therefore by \eqref{e-5.5} and \eqref{e-5.2},
$$
\mu\circ \pi^{-1}([J]) \approx e^{-|J|h_{\rm top}(Y)}\approx \nu([J])   \quad \mbox{ for }J\in \L(Y).
$$
 By Theorem \ref{thm-Parry}, we have $\mu\circ \pi^{-1}=\nu$. This completes the proof of the theorem.
\end{proof}
\begin{rem}
 Theorem \ref{thm-5.1} was  partially proved in the second author's master thesis \cite{Lo2012}.
\end{rem}

\section{Final remarks and questions}
\label{S-8}
In this section we give a few  more remarks.

  First we remark that without any assumption of irreducibility, there is no  algorithm to check whether a given tuple ${\bf M}$ of square matrices has a uniform Lyapunov exponent modulo $0$. This fact was first  pointed out in \cite[Theorem 8]{ProtasovVoynov2014} in a different context. Indeed, let $A_1,\ldots, A_k$ be a finite family of $n\times n$ non-negative matrices with rational entries and $\rho(A_1+\cdots+ A_k)\leq k$. Set ${\bf M}=(M_1,\ldots, M_k)$ by
$$
M_i=\left(\begin{array}{cc}
1 & 0\\
0 & A_i
\end{array}
\right),\quad i=1,\ldots, k.
$$
It is easy to see that ${\bf M}$ is normalized.   Moreover, ${\bf M}$ has a uniform Lyapunov exponent modulo $0$ if and only if the semigroup generated by $\{A_1,\ldots, A_k\}$ is bounded. However, as proved by Blondel and Tsitsiklis \cite{BlondelTsitsiklis2000},  the problem of determining whether the semigroup generated by a finite set of non-negative matrices with rational entries is bounded, is in general arithmetically undecidable. Hence, there is no  algorithm to check whether the constructed  ${\bf M}$ has a uniform Lyapunov exponent modulo $0$ within finite time.

We also remark that in Theorem \ref{thm-1.6}, the irreducibility (resp. positively irreducibility) assumption on ${\bf M}$ can be replaced by a more general assumption: there exist $C>0$ and $m\in \N$ such that 
\begin{equation}
\label{e-multi}
\sum_{K\in \A^*:\; |K|\leq m}\|M_{IKJ}\|\geq C\|M_I\|\|M_J\| \qquad \mbox{ for all }I, J\in \A^*.
\end{equation}
Indeed under the above condition, the conclusion of Theorem \ref{thm-matrix} still holds (see \cite[Theorem 5.5]{Feng2011}) and the proof of Theorem \ref{thm-1.6} remains valid. It is a natural problem to decide whether a given tuple ${\bf M}$ satisfies the condition \eqref{e-multi} for some $C$ and $m$.

Next we present an extended version of Question \ref{ques-1}. Let $(X,T)$ be a topological dynamical system, that is, $X$ is a compact metric space and $T: X\to X$ a continuous transformation.  Let $M$ be a Borel function on $X$ taking values in the set of real (or complex) $d\times d$ matrices.
\begin{de}We say that $M$ has a uniform Lyapunov exponent on $(X, T)$ if there exists $\lambda\in \R$ such that
$$
\|M(n,x)\|\approx e^{\lambda n},\quad n\in \N, \; x\in X,
$$
where $M(n,x):=M(x)M(Tx)\cdots M(T^{n-1}x)$.
\end{de}

For a given tuple ${\bf M}=(M_1,\ldots, M_k)$ of non-negative matrices, defining
\begin{equation}
\label{e-last}M(x)=M_{x_1} \mbox{ for }x=(x_n)_{n=1}^\infty\in \{1,\ldots, k\}^\N,
\end{equation}
 we see that
${\bf M}$ has a uniform Lyapunov exponent modulo $0$ if and only if that $M$ has a  uniform Lyapunov exponent on $(Y_{\bf M},\sigma)$, where $Y_{\bf M}$ is defined as in \eqref{e-YM}.

As a general extension of Question \ref{ques-1}, one may ask under which condition, a matrix-valued function  $M$ on a given  topological dynamical system
$(X, T)$  has a uniform Lyapunov exponent on $(X, T)$ and how to check it.

 In the end of this paper, we mention a particular example of the above general question. Let ${\bf M}=(M_1,\ldots, M_k)$ be a tuple  of non-negative $d\times d$ matrices and let $\Sigma_A$ be an irreducible subshift of finite type over the alphabet $\{1,\ldots, k\}$. Let $M$ be the matrix-valued
function defined as in \eqref{e-last}. We remark that in this setting, the preceding assumption of positive irreducibility on ${\bf M}$ is no longer sufficient to guarantee that one can check whether $M$ has a uniform Lyapunov exponent on $(\Sigma_A,\sigma)$. Nevertheless, the following stronger assumption on ${\bf M}$ (acting on $\Sigma_A$)   is enough for providing an affirmative answer to the deciding problem: for any $i, i'\in \{1,\ldots, d\}$ and $j, j'\in \{1,\ldots, k\}$, there exists  a finite word $J$ such that $jJj'\in {\mathcal L}(\Sigma_A)$ and $(M_J)_{i,i'}>0$. The justification is quite similar to that of Theorem \ref{thm-1.4}. The details of the proof and the counter example will be included in the Ph.D.~thesis of the second author.

\medskip

\noindent {\bf Acknowledgement}. This research was conducted as part of the second author's Ph.D.~studies. The first author was partially supported by the RGC grant in Hong Kong. The third author was supported by the RGC grant in Hong Kong and the postdoc fellowship in CUHK.  The authors thank Wen Huang for pointing out Lemma \ref{lem-4.6} and  Karoly Simon for pointing out the reference \cite{Ruiz2008}.  They are grateful to Ian Morris for giving many helpful comments and pointing out some related references.


\end{document}